\documentclass[11pt,oneside]{amsart}
\usepackage{amsthm}	
\usepackage[latin9]{inputenc}
\usepackage{fullpage}
\usepackage{enumitem}
\usepackage{stmaryrd} 
\usepackage{mathtools}
\newtheorem{thm}{Theorem}
\newtheorem{lem}[thm]{Lemma}
\newtheorem{prop}[thm]{Proposition}
\newtheorem{defn}[thm]{Definition}
\newtheorem{cor}[thm]{Corollary}

\newtheorem{conjecture}[thm]{Conjecture}
\newtheorem{rem}{Remark}
\newcounter{cl}
\newtheorem{claim}[cl]{Claim}
\def\ACC{\mathbf{ACC}_{\mathrm{w}^*}} 
\def\LIM{\mathbf{LIM}_{\mathrm{w}^*}} 
\def\INT{\mathrm{INT}} 
\def\d{\mathrm{d}} 
\def\EXP{\mathbb{E}} \def\PROB{\mathbb{P}} 
\newcommand{\JUSTIFY}[1]{\fbox{\tiny{#1}}\quad} 

\begin{document}
\title{Cut-norm and entropy minimization over weak$^*$ limits}
\author{Martin Dole\v zal}
\author{Jan Hladk\'y}
\address{Institute of Mathematics, Czech Academy of Sciences. \v Zitn\'a 25, 110 00, Praha, Czech Republic. The Institute of Mathematics of the Czech Academy of Sciences is supported by RVO:67985840.}
\email{dolezal|hladky@math.cas.cz}
\thanks{Jan Hladk\'y was supported by the Alexander von Humboldt Foundation. Research of Martin Dole\v zal was supported by the GA\v CR project GA16-07378S.}

\begin{abstract}
We prove that the accumulation points of a sequence of graphs $G_1,G_2,G_3,\ldots$ with respect to the cut-distance are exactly the weak$^*$ limit points of subsequences of the adjacency matrices (when all possible orders of the vertices are considered) that minimize the entropy over all weak$^*$ limit points of the corresponding subsequence. In fact, the entropy can be replaced by any map $W\mapsto \iint f(W(x,y))$, where $f$ is a continuous and strictly concave function. 
As a corollary, we obtain a new proof of compactness of the cut-distance topology.
\end{abstract}

\maketitle

\section{Introduction}
The theory of limits of dense graphs was developed in~\cite{Lovasz2006,Borgs2008c} and has revolutionized graph theory since then. The key objects of the theory are so-called graphons. More precisely, a graphon is a symmetric Lebesgue measurable function from $I^2$ to $[0,1]$ where $I=[0,1]$ is the unit interval (equipped by the Lebesgue measure $\lambda$). In the heart of the theory is then the following statement.
\begin{thm}[Informally]\label{thm:converge}
Suppose that $G_1,G_2,G_3,\ldots$ is a sequence of graphs. Then there exists a subsequence $G_{k_1},G_{k_2},G_{k_3},\ldots$ and a graphon $W:I^2\rightarrow [0,1]$ such that $G_{k_1},G_{k_2},G_{k_3},\ldots$ converges to $W$.
\end{thm}
Roughly speaking, to obtain the graphon $W$ one looks at the adjacency matrices of the graphs $(G_{k_n})_n$ from distance. One possible way an analyst might attempt to make this statement formal could be to take $W$ as a weak$^*$ limit\footnote{See the Appendix for basic information about the weak$^*$  topology.} of adjacency matrices of the graphs $(G_{k_n})_n$ represented as functions from $I^2$ to $\{0,1\}$. Such a version of Theorem~\ref{thm:converge} would be just an instance of the Banach--Alaoglu Theorem. However, the weak$^*$ topology turns out to be too coarse to provide the favorable properties that are available in the contemporary theory of graph limits.\footnote{A primal example of such a favorable  property is the continuity of subgraph densities.} A good toy example is the sequence of the complete balanced bipartite graphs $(K_{n,n})_{n=1}^\infty$. When considering adjacency matrices of these graphs with vertices grouped into the two parts of the bipartite graphs, the corresponding weak$^*$ limit is a $2\times 2$-chessboard function with values $0$ and $1$, which we denote by $W_{\mathrm{bipartite}}$. This turns out to be a desirable limit. On the other hand, one could consider adjacency matrices ordered differently. Ordering the vertices randomly, we get the constant $W_{\mathrm{const}}\equiv\frac12$ as the weak$^*$ limit (almost surely). We see that it is undesirable to get $W_{\mathrm{const}}$ as the limit object as the only information carried by such an object is that the overall edge densities of the graphs along the sequence converge to $\frac12$.

So, instead of the weak$^*$ topology one considers the so-called cut-norm topology, and this is also the topology to which ``converges to $W$'' in  Theorem~\ref{thm:converge} refers.
The cut-norm $\|\cdot\|_\square$ is a certain uniformization of the weak$^*$ topology. Indeed, recall that given symmetric measurable functions $\Gamma:I^2\rightarrow[0,1]$ and $\Gamma_1,\Gamma_2,\Gamma_3,\ldots:I^2\rightarrow[0,1]$, the two convergence notions compare as follows. 
\begin{align*}
\Gamma_n \overset{\mathrm{w}^*}\longrightarrow  \Gamma \qquad&\Longleftrightarrow\qquad
\sup_{B\subset I}\left\{\limsup_n
\left|\int_{x\in B}\int_{y\in B} \Gamma_n(x,y)-\Gamma(x,y)\right|\right\}=0\;,\\
\Gamma_n \overset{\|\cdot\|_\square}\longrightarrow  \Gamma \qquad&\Longleftrightarrow\qquad
\limsup_n\left\{\sup_{B\subset I}
\left|\int_{x\in B}\int_{y\in B} \Gamma_n(x,y)-\Gamma(x,y)\right|\right\}=0\;.
\end{align*}
We shall state the formal version of Theorem~\ref{thm:converge} in a somewhat bigger generality for graphons. If $\Gamma,\Gamma':I^2\rightarrow[0,1]$ are two graphons then we say that they are \emph{versions} of each other if they differ only by some measure-preserving transformation of $I$ (see Section~\ref{sec:Notation} for a precise definition).

Then the formal statement of Theorem~\ref{thm:converge} reads as follows.
\begin{thm}\label{thm:convergeFormal}
Suppose that $\Gamma_1,\Gamma_2,\Gamma_3,\ldots:I^2\rightarrow [0,1]$ is a sequence of graphons. Then there exists a sequence $k_1<k_2<k_3<\cdots$ of natural numbers, versions $\Gamma'_{k_1},\Gamma'_{k_2},\Gamma'_{k_3},\ldots$ of $\Gamma_{k_1},\Gamma_{k_2},\Gamma_{k_3},\ldots$, and a graphon $W:I^2\rightarrow [0,1]$ such that the sequence $\Gamma'_{k_1},\Gamma'_{k_2},\Gamma'_{k_3},\ldots$  converges to $W$ in the cut-norm.
\end{thm}

Prior to our work, there were three approaches to proving Theorem~\ref{thm:convergeFormal}. One, taken in~\cite{Lovasz2006} and in~\cite{Lovasz2007}, uses (variants of) the regularity lemma to group parts of $I$ according to the structure of $\Gamma_n$. This way, one approximates the graphons by step-functions, and the limit graphon $W$ is a limit of these step-functions.\footnote{A very general compactness result was given by Regts and Schrijver,~\cite{MR3425986}. This result in particular subsumes the compactness of the  graphon space. Even when specialized to the space of graphons, there are differences of debatable significance between the proofs.} A second approach, taken in~\cite{ElekSzegedy}, relies on ultraproduct techniques. This later approach is extremely technical, and was developed for the (more difficult) theory of limits of hypergraphs, where for some time the regularity approach was not available.\footnote{A regularity approach to hypergraph limits was later found by Zhao,~\cite{MR3382671}.} The third proof follows from the Aldous--Hoover theorem for exchangeable arrays (\cite{MR637937}). While the Aldous--Hoover theorem substantially precedes the theory of graph limits, the connection was realized substantially later by Diaconis and Janson, \cite{MR2463439} and independently by Austin~\cite{MR2426176}.

We present a fourth proof of Theorem~\ref{thm:convergeFormal}. Our proof provides for the first time a characterization of the cut-norm convergence in terms of the weak$^*$ convergence. Namely, fixing any continuous and strictly concave function $f:[0,1]\rightarrow\mathbb R$, we prove that there is a subsequence $\Gamma_{k_1},\Gamma_{k_2},\Gamma_{k_3},\ldots$ such that the map $W\mapsto \iint f(W(x,y))$ attains its minimum on the space of all weak$^*$ accumulation points of versions of graphons $\Gamma_{k_1},\Gamma_{k_2},\Gamma_{k_3},\ldots$, and that any such minimizer is an accumulation point of the sequence $\Gamma_1,\Gamma_2,\Gamma_3,\ldots$ in the cut-distance. This result is consistent with our toy example above. Indeed, for any strictly concave function $f$ we have $\iint f\left(W_{\mathrm{bipartite}}(x,y)\right)<\iint f\left(W_{\mathrm{const}}(x,y)\right)$ by Jensen's inequality. Jensen's inequality underlies the general proof of our result. 
This application of Jensen's inequality is in a sense analogous to the proof of the index-pumping lemma in proofs of the regularity lemma. We try to indicate this important link in Section~\ref{ssec:connectionwithRL}.

\subsection{Statement of the main results}
Let $f:[0,1]\rightarrow \mathbb R$ be an arbitrary continuous and strictly concave function.\footnote{The additional assumption of the continuity of the concave function $f:[0,1]\rightarrow\mathbb R$ which we work with in this paper only means that $f$ is continuous (from the appropriate sides) at 0 and at 1, so this is not a big extra restriction.} Given a graphon $\Gamma:I^2\rightarrow[0,1]$, we write $\INT_f(\Gamma):=\int_{x\in I}\int_{y\in I}f(\Gamma(x,y))$. When $f$ is the binary entropy, the integration $\INT_f(W)$ appears also in the work on large deviations in random graphs,~\cite{ChatVar:LargeDev} (which does not relate to the current work otherwise), and is called the \emph{entropy of the graphon $W$}.\footnote{As was pointed out to us by Svante Janson, Aldous (\cite[p. 145]{MR883646}) worked with this quantity already in the 1980's in the context of exchangeability.}

For a sequence $\Gamma_1,\Gamma_2,\Gamma_3,\ldots:I^2\rightarrow[0,1]$ of graphons, we denote by $\ACC(\Gamma_1,\Gamma_2,\Gamma_3,\ldots)$ the set of all functions $W:I^2\rightarrow [0,1]$ for which there exist versions $\Gamma'_1,\Gamma'_2,\Gamma'_3,\ldots$ of $\Gamma_1,\Gamma_2,\Gamma_3,\ldots$ such that $W$ is a weak$^*$ accumulation point of the sequence $\Gamma'_1,\Gamma'_2,\Gamma'_3,\ldots$. We also denote by $\LIM(\Gamma_1,\Gamma_2,\Gamma_3,\ldots)$ the set of all functions $W:I^2\rightarrow [0,1]$ for which there exist versions $\Gamma'_1,\Gamma'_2,\Gamma'_3,\ldots$ of $\Gamma_1,\Gamma_2,\Gamma_3,\ldots$ such that $W$ is a weak$^*$ limit of the sequence $\Gamma'_1,\Gamma'_2,\Gamma'_3,\ldots$. We have $\LIM(\Gamma_1,\Gamma_2,\Gamma_3,\ldots)\subset\ACC(\Gamma_1,\Gamma_2,\Gamma_3,\ldots)$. Note that $\LIM(\Gamma_1,\Gamma_2,\Gamma_3,\ldots)$ can be empty but $\ACC(\Gamma_1,\Gamma_2,\Gamma_3,\ldots)$ cannot be empty by the sequential Banach--Alaoglu Theorem (see the Appendix for more details). Also, note that such weak$^*$ accumulation points (and thus also limits) are necessarily symmetric, Lebesgue measurable, $[0,1]$-valued, and thus graphons.

Our main result states that, given a sequence of graphons $\Gamma_1,\Gamma_2,\Gamma_3,\ldots$, there is a subsequence $\Gamma_{k_1},\Gamma_{k_2},\Gamma_{k_3},\ldots$ such that the minimum of $\INT_f(\cdot)$ over the set  $\ACC(\Gamma_{k_1},\Gamma_{k_2},\Gamma_{k_3},\ldots)$ is attained, and the graphon attaining this minimum is an accumulation point of the sequence $\Gamma_1,\Gamma_2,\Gamma_3,\ldots$ in the cut-distance.

\begin{thm}\label{thm:compactANDminimizer}
Suppose that $f:[0,1]\rightarrow \mathbb R$ is an arbitrary continuous and strictly concave function.
Suppose that $\Gamma_1,\Gamma_2,\Gamma_3,\ldots:I^2\rightarrow[0,1]$ is a sequence of graphons. 
\begin{enumerate}[label=(\alph*)]
\item\label{en:improve} Suppose that $W\in\ACC(\Gamma_1,\Gamma_2,\Gamma_3,\ldots)$ is not an accumulation point of 
the sequence $\Gamma_1,\Gamma_2,\Gamma_3,\ldots$ in the cut-norm. Then there exists $\widetilde W\in\ACC(\Gamma_1,\Gamma_2,\Gamma_3,\ldots)$ such that $\INT_f(\widetilde W)<\INT_f(W)$.
\item\label{en:compact} There exist a subsequence $\Gamma_{k_1},\Gamma_{k_2},\Gamma_{k_3},\ldots$ and a graphon $W_{\min}\in\ACC(\Gamma_{k_1},\Gamma_{k_2},\Gamma_{k_3},\ldots)$ such that
\begin{equation}\nonumber
\INT_f(W_{\min})=\inf\left\{\INT_f(W)\colon W\in\ACC(\Gamma_{k_1},\Gamma_{k_2},\Gamma_{k_3},\ldots)\right\}\;.
\end{equation}
\end{enumerate}
\end{thm}
Clearly, Theorem~\ref{thm:compactANDminimizer} implies Theorem~\ref{thm:convergeFormal}.

The proof of Theorem~\ref{thm:compactANDminimizer} is given in Sections~\ref{sec:proofImprove} and~\ref{sec:proofCompact}.

To complete the ``characterization of the cut-norm convergence in terms of the weak$^*$ convergence'' advertised above, we prove that weak$^*$ limit points that do not minimize $\INT_f(\cdot)$ cannot be limit points in the cut-norm.
\begin{prop}\label{prop:cutnormMIN}
Suppose that $f:[0,1]\rightarrow \mathbb R$ is an arbitrary continuous and strictly concave function.
Suppose that $\Gamma_1,\Gamma_2,\Gamma_3,\ldots:I^2\rightarrow[0,1]$ is a sequence of graphons. If $W\in\LIM(\Gamma_1,\Gamma_2,\Gamma_3,\ldots)$ is a cut-norm limit of versions of $\Gamma_1,\Gamma_2,\Gamma_3,\ldots$ then $W$ is a minimizer of $\INT_f(\cdot)$
over the space $\LIM(\Gamma_1,\Gamma_2,\Gamma_3,\ldots)$.
\end{prop}
In Section~\ref{prop:cutnormMIN} we show that Proposition~\ref{prop:cutnormMIN} is an easy consequence of a result of Borgs, Chayes, and Lov\'asz~\cite{MR2594615} on uniqueness of graph limits. In addition, we give a self-contained proof.

\section{Notation and tools}\label{sec:Notation}
For every function $W:I^2\rightarrow\mathbb R$, we define the \emph{cut-norm} of $W$ by
\begin{equation}\label{def:cut-norm}
\|W\|_\square=\sup_A\left|\int_A\int_AW(x,y)\right|\;,
\end{equation}
where $A$ ranges over all measurable subsets of $I$. Another slightly different formula is also often used in the literature where one replaces the right-hand side of (\ref{def:cut-norm}) by $\sup_{A,B}\left|\int_A\int_BW(x,y)\right|$
where two sets $A$ and $B$ range over all measurable subsets of $I$. However, it is easy to see that for every symmetric function $W$, we have
$$\sup_{A,B}\left|\int_A\int_BW(x,y)\right|\ge \sup_A\left|\int_A\int_AW(x,y)\right|\ge\frac12\sup_{A,B}\left|\int_A\int_BW(x,y)\right|,$$
and so the notion of convergence of sequences of graphons (which are symmetric) in the cut-norm is irrelevant to the choice between these two formulas.

We say that a graphon $\Gamma\colon I^2\rightarrow [0,1]$ is a \emph{step-graphon} with steps $I_1,I_2,\ldots,I_k\subset I$ if the sets $I_1,I_2,\ldots,I_k$ are pairwise disjoint, $I_1\cup I_2\cup\ldots\cup I_k=I$ and $W_{|I_i\times I_j}$ is constant (up to a null set) for every $i,j=1,2,\ldots,k$.

We say that a measurable function $\gamma:I\rightarrow I$ is an \emph{almost-bijection} if there exist conull sets $J_1,J_2\subset I$ such that $\gamma_{|J_1}$ is a bijection from $J_1$ onto $J_2$. When we talk about the inverse of such a function $\gamma$ then we mean $(\gamma_{|J_1})^{-1}$ but we denote it only by $\gamma^{-1}$. Note that this inverse $\gamma^{-1}$ is not unique but that does not cause any problems as any two inverses of $\gamma$ differ only on a null set.

If $\Gamma,\Gamma':I^2\rightarrow[0,1]$ are two graphons then we say that $\Gamma'$ is a \emph{version} of $\Gamma$ if there exists a measure preserving almost-bijection $\gamma:I\rightarrow I$ such that $\Gamma'(x,y)=\Gamma(\gamma^{-1}(x),\gamma^{-1}(y))$ for almost every $(x,y)\in I^2$.

Related to versions, we recall that the \emph{cut-distance} and \emph{$L^1$-distance} between two graphons $W_1,W_2$ are defined as $\delta_\square(W_1,W_2)=\inf \|U_1-W_2\|_\square$ and $\delta_1(W_1,W_2)=\inf \|U_1-W_2\|_1$ where $U_1$ ranges over all versions of $W_1$.

By an \emph{ordered partition} of $I$, we mean a partition of $I$ with a fixed order of the sets from the partition.
For an ordered partition $\mathcal J$ of $I$ into finitely many sets $C_1,C_2,\ldots,C_k$, we define mappings $\alpha_{\mathcal J,1},\alpha_{\mathcal J,2},\ldots,\alpha_{\mathcal J,k}:I\rightarrow I$, and a mapping $\gamma_{\mathcal J}:I\rightarrow I$ by
\begin{equation}
\begin{split}\label{eq:pocitac}
\alpha_{\mathcal J,1}(x)&=
\int_0^x{\bf 1}_{C_1}(y)\:\d(y)\;,\\
\alpha_{\mathcal J,2}(x)&=\alpha_{\mathcal J,1}(1)+\int_0^x{\bf 1}_{C_2}(y)\:\d(y)\;,\\
\vdots\\
\alpha_{\mathcal J,k}(x)&=\alpha_{\mathcal J,1}(1)+\alpha_{\mathcal J,2}(1)+\ldots+\alpha_{\mathcal J,k-1}(1)+\int_0^x{\bf 1}_{C_k}(y)\:\d(y)\;,\\
\gamma_{\mathcal J}(x)&=
\alpha_{\mathcal J,i}(x)\quad\text{if }x\in C_i,\quad i=1,2,\ldots,k\;.
\end{split}
\end{equation}
Informally, $\gamma_{\mathcal J}$ is defined in such a way that it maps the set $C_1$ to the left side of the interval $I$, the set $C_2$ next to it, and so on. Finally, the set $C_k$ is mapped to the right side of the interval $I$. Clearly, $\gamma_{\mathcal J}$ is a measure preserving almost-bijection.

For a graphon $W:I^2\rightarrow[0,1]$ and an ordered partition $\mathcal J$ of $I$ into finitely many sets, we denote by $\prescript{}{\mathcal J}W$ the version of $W$ defined by $\prescript{}{\mathcal J}W(x,y)=W(\gamma_{\mathcal J}^{-1}(x),\gamma_{\mathcal J}^{-1}(y))$ for every $(x,y)\in I^2$.

\subsection{Lebesgue points}
The Lebesgue density theorem asserts that given an integrable function $f:\mathbb{R}^n\rightarrow \mathbb{R}$, almost every point $x\in\mathbb{R}^n$ is a Lebesgue point of $f$, meaning that the value of $f(x)$ equals to the limit of the averages of $f$ on neighborhoods of $x$ of diminishing sizes. There is some freedom in choosing the particular shapes of these neighborhoods. Below, we give a definition of Lebesgue points tailored to our purposes. Since we shall work with graphons, we state this definition for the domain $I^2$.
\begin{defn}
Suppose that $W:I^2\rightarrow\mathbb R$ is an integrable function. We say that $(x,y)\in I^2$ is a \emph{Lebesgue point} of $W$ if for every $\eta>0$ there exists $\delta_0>0$ such that whenever $[p_1,p_2]\subset I$ and $[q_1,q_2]\subset I$ are intervals such that the length of the intervals is smaller or equal to $\delta_0$, such that the ratio of the lengths of these intervals is at least $\tfrac 12$ and at most 2, and such that $[p_1,p_2]$ contains $x$ and $[q_1,q_2]$ contains $y$ then
\begin{equation}\label{eq:LebPAsWeUseIt}
\left|W(x,y)-\frac 1{(p_2-p_1)(q_2-q_1)}\int_{p_1}^{p_2}\int_{q_1}^{q_2}W(w,z)\:\d(w)\:\d(z)\right|<\eta\;.
\end{equation}
\end{defn}
We can now state the Lebesgue density theorem.
\begin{thm}[Lebesgue density theorem]
	Suppose that $W:I^2\rightarrow\mathbb R$ is an integrable function. Then almost every point of $I^2$ is a Lebesgue point of $W$.
\end{thm}

\subsection{Stepping}

The next definition introduces graphons derived by an averaging of a given graphon $W$ on a given partition of $I$. Here, we denote by $\lambda^{\oplus 2}$ the two-dimensional Lebesgue measure on $I^2$.

\begin{defn}
	Suppose that $W:I^2\rightarrow [0,1]$ is a graphon. For a partition $\mathcal I$ of the unit interval into finitely many sets of positive measure, $I=I_1\sqcup I_2\sqcup \ldots \sqcup I_k$, we define a \emph{stepping} $W^{\Join\mathcal I}$ which is defined on each rectangle $I_i\times I_j$ to be the constant $\frac 1{\lambda^{\oplus 2}(I_i\times I_j)}\int_{I_i}\int_{I_j}W(x,y)$.
\end{defn}

The next lemma shows that we can replace any graphon $W$ by its stepping (on some partition of $I$) without changing the value of $\INT_f(W)$ too much.

\begin{lem}\label{lem:aproximacePrumerama}
	Let $f:[0,1]\rightarrow \mathbb R$ be an arbitrary continuous and strictly concave function, and let $\mathcal J$ be an arbitrary partition of $I$ into finitely many intervals of positive measure. Suppose that $W\colon I^2\rightarrow[0,1]$ is a graphon, and let $\varepsilon>0$. Then there exists a partition $\mathcal I$ of $I$ into finitely many intervals of positive measure such that $\mathcal I$ is a refinement of $\mathcal J$ and such that $|\INT_f(W)-\INT_f(W^{\Join\mathcal I})|<\varepsilon$.
\end{lem}

\begin{proof}
	As $f$ is continuous, there is $\eta>0$ such that $|f(x)-f(y)|<\tfrac 12\varepsilon$ whenever $x,y\in [0,1]$ are such that $|x-y|<\eta$. Also, as $W$ is an integrable function, almost every point $(x,y)\in I^2$ is a Lebesgue point of $W$. This implies that for a.e. $(x,y)\in I^2$ there is a natural number $n$ such that whenever $[p_1,p_2]\subset I$ and $[q_1,q_2]\subset I$ are intervals of lengths smaller or equal to $\tfrac 2n$ such that the ratio of the lengths is at least $\tfrac 12$ and at most 2, and such that $[p_1,p_2]$ contains $x$ and $[q_1,q_2]$ contains $y$ then inequality~(\ref{eq:LebPAsWeUseIt}) holds. For every such $(x,y)$, we denote by $n(x,y)$ the smallest $n$ with this property. For every natural number $n$, we also put
	\begin{equation}
	\nonumber
	D_n=\left\{(x,y)\in I^2\colon n(x,y)>n\right\}\;.
	\end{equation}
	Then it is easy to check that the sets $D_1\supseteq D_2\supseteq D_3\supseteq\ldots$ are measurable and $\lambda^{\oplus 2}\left(\bigcap_{n=1}^\infty D_n\right)=0$. So, after denoting $C:=\max_{x\in [0,1]}|f(x)|$, we can find a natural number $n_0$ large enough such that
	\begin{equation}\label{eq:measureOfB}
	\lambda^{\oplus 2}\left(D_{n_0}\right)<\frac 1{4C}\varepsilon\;,
	\end{equation}
	and such that $\tfrac 1{n_0}$ is smaller than the length of all intervals from the partition $\mathcal J$. Now let $\mathcal I$ be an arbitrary refinement of the partition $\mathcal J$ into finitely many intervals $I_1,I_2,\ldots,I_k$, such that the length of each of these intervals is at least $\tfrac 1{n_0}$ and at most $\tfrac 2{n_0}$. For each $i,j=1,2,\ldots,k$, we denote $C_{i,j}=\frac 1{\lambda^{\oplus 2}(I_i\times I_j)}\int_{I_i}\int_{I_j}W(x,y)$. Inequality~(\ref{eq:LebPAsWeUseIt}) then tells us that
	\begin{equation}\nonumber
	\left|W(x,y)-C_{i,j}\right|<\eta\qquad\text{for every }(x,y)\in (I_i\times I_j)\setminus D_{n_0},\quad i,j=1,2,\ldots,k\;,
	\end{equation}
	and so
	\begin{equation}\label{eq:aproximaceHodnot_f(W)}
	\left|f(W(x,y))-f(C_{i,j})\right|<\tfrac 12\varepsilon\qquad\text{for every }(x,y)\in (I_i\times I_j)\setminus D_{n_0},\quad i,j=1,2,\ldots,k\;.
	\end{equation}
	So we have
	\begin{equation}\nonumber
	\begin{aligned}
	&\left|\INT_f(W)-\INT_f(W^{\Join\mathcal I})\right|\\
	\le&\iint\limits_{D_{n_0}}\left|f(W(x,y))-f(W^{\Join\mathcal I}(x,y))\right|+\sum_{i,j=1}^k\;\iint\limits_{(I_i\times I_j)\setminus D_{n_0}}\left|f(W(x,y))-f(W^{\Join\mathcal I}(x,y))\right|\\
	\stackrel{(\ref{eq:aproximaceHodnot_f(W)})}{\le}&2C\cdot\lambda^{\oplus 2}(D_{n_0})+\frac 12\varepsilon \sum_{i,j=1}^k\lambda^{\oplus 2}\left((I_j\times I_j)\setminus D_{n_0}\right)\\
	\stackrel{(\ref{eq:measureOfB})}{<}&\frac 12\varepsilon+\frac 12\varepsilon=\varepsilon\;,
	\end{aligned}
	\end{equation}
	as we wanted.
\end{proof}
The next lemma says that if a graphon is a weak$^*$ limit point then so is any graphon derived by an averaging of the original one on a given partition of $I$ into intervals.
\begin{lem}\label{lem:attainaveraged}
	Suppose that $\Gamma_1,\Gamma_2,\Gamma_3,\ldots:I^2\rightarrow[0,1]$ is a sequence of graphons. Suppose that $W\in\LIM(\Gamma_1,\Gamma_2,\Gamma_3,\ldots)$ and that we have a partition $\mathcal I$ of $I$ into finitely many intervals of positive measure.
	Then $W^{\Join\mathcal I}\in\LIM(\Gamma_1,\Gamma_2,\Gamma_3,\ldots)$.
	
	Moreover, whenever $\Gamma'_1,\Gamma'_2,\Gamma'_3,\ldots$ are versions of $\Gamma_1,\Gamma_2,\Gamma_3,\ldots$ which converge to $W$ in the weak$^*$ topology then the versions $\Gamma''_1,\Gamma''_2,\Gamma''_3,\ldots$ of $\Gamma_1,\Gamma_2,\Gamma_3,\ldots$ weak$^*$ converging to $W^{\Join\mathcal I}$ can be chosen in such a way that for every natural number $j$ and for every intervals $K,L\in\mathcal I$ it holds
	$$\int_K\int_L\Gamma_j'(x,y)=\int_K\int_L\Gamma_j''(x,y)\;.$$
\end{lem}
The proof of Lemma~\ref{lem:attainaveraged} follows a relatively standard probabilistic argument. Suppose for simplicity that $\Gamma_1,\Gamma_2,\Gamma_3,\ldots$ weak$^*$ converges to $W$. Then, for each $n$, we consider a version $\Gamma'_n$ of $\Gamma_n$ which is obtained by splitting each interval  $A\in \mathcal I$ into $n$ subsets of the same measure and then permuting these subsets of $A$ at random. It can then be shown that $\Gamma'_1,\Gamma'_2,\Gamma'_3,\ldots$ converge to $W^{\Join\mathcal I}$ almost surely. The next two definitions are needed to make precise the notion of randomly permuting parts of the graphon within a given partition.
\begin{defn}
	Given a set $A\subset I$ of positive measure and a number $s\in\mathbb{N}$, we can consider a partition $A=\llbracket A \rrbracket^s_1\sqcup \llbracket A \rrbracket^s_2\sqcup \ldots \sqcup \llbracket A \rrbracket^s_s$, where each set $\llbracket A \rrbracket^s_i$ has measure $\frac{\lambda(A)}s$ and for each $1\le i<j\le s$, the set $\llbracket A \rrbracket^s_i$ is entirely to the left of $\llbracket A \rrbracket^s_j$. These conditions define the partition $A=\llbracket A \rrbracket^s_1\sqcup \llbracket A \rrbracket^s_2\sqcup \ldots \sqcup \llbracket A \rrbracket^s_s$ uniquely, up to null sets. For each $i,j\in [s]$ there is a natural, uniquely defined (up to null sets), measure preserving almost-bijection $\chi^{A,s}_{i,j}:\llbracket A\rrbracket^s_i\rightarrow \llbracket A\rrbracket^s_j$ which preserves the order on the real line.
\end{defn}
\begin{defn}\label{def:chy}
	Suppose that $\Gamma:I^2\rightarrow [0,1]$ is a graphon. For a partition $\mathcal I$ of $I$ into finitely many sets of positive measure, $I=I_1\sqcup I_2\sqcup \ldots \sqcup I_k$, and for $s\in\mathbb{N}$, we define a discrete distribution $\mathbb{W}(\Gamma,\mathcal{I},s)$ on graphons using the following procedure. We take $\pi_1,\ldots,\pi_k:[s]\rightarrow [s]$ independent uniformly random permutations. After these are fixed, we define a sample $W\sim\mathbb{W}(\Gamma,\mathcal{I},s)$ by 
	$$W(x,y)=\Gamma\left(\chi^{I_i,s}_{p,\pi_i(p)}(x),\chi^{I_j,s}_{q,\pi_j(q)}(y)\right)\quad \mbox{when $x\in \llbracket I_i\rrbracket^s_p$, $y\in \llbracket I_j\rrbracket^s_q$, $i,j\in [k]$, $p,q\in[s]$}\;.$$
	This defines the sample $W:I^2\rightarrow [0,1]$ uniquely up to null sets, and thus defines the whole distribution $\mathbb{W}(\Gamma,\mathcal{I},s)$. Observe that $\mathbb{W}(\Gamma,\mathcal{I},s)$ is supported on (some) versions of $\Gamma$.
	
	We call the sets $\llbracket I_j\rrbracket^s_q$ \emph{stripes}.
\end{defn}
\begin{proof}[Proof of Lemma~\ref{lem:attainaveraged}]
	By considering suitable versions of the graphons $\Gamma_n$, we can without loss of generality assume that the sequence $\Gamma_1,\Gamma_2,\Gamma_3,\ldots$ itself converges to $W$ in the weak$^*$ topology. For each $n\in\mathbb{N}$, let us sample $U_n\sim \mathbb{W}(\Gamma_n,\mathcal{I},n)$. We claim that the sequence $U_1,U_2,U_3,\ldots$ converges to $W^{\Join\mathcal I}$ in the weak$^*$ topology almost surely. As each $U_n$ is a version of $\Gamma_n$, this will prove the lemma. So, let us now turn to proving the claim.
	
	Let $i,j\in[k]$ be arbitrary. Further, let $0\le p_1<p_2\le 1$ and $0\le r_1<r_2\le 1$ be arbitrary rational numbers such that the rectangle $[p_1,p_2]\times[r_1,r_2]$ is contained (modulo a null set) in $I_i\times I_j$. Having fixed $i,j,p_1,p_2,r_1,r_2$, let us write $c$ for the value of $W^{\Join\mathcal I}$ on $I_i\times I_j$. For each $n\in\mathbb{N}$, let $E_n$ be the event that 
	$$\left|\iint_{[p_1,p_2]\times[r_1,r_2]}U_n\d(\lambda^{\oplus2})-c(p_2-p_1)(r_2-r_1)\right|>\sqrt[4]{1/n}+\tfrac 4n\;.$$
	Let us now bound the probability that $E_n$ occurs. To this end, let $Y_n$ be the value of $\iint_{[p_1,p_2]\times[r_1,r_2]}U_n\d(\lambda^{\oplus2})$. We clearly have $\EXP[Y_n]=c(p_2-p_1)(r_2-r_1)\pm\tfrac 4{n}$ (the error $\pm\tfrac 4{n}$ comes from those products of pairs of stripes that intersect both $[p_1,p_2]\times[r_1,r_2]$ and its complement). Therefore, if $E_n$ occurs then $|Y_n-\EXP[Y_n]|>\sqrt[4]{1/n}$. Suppose that we want to compute $Y_n$. From the $k$ random permutations $\pi_1,\pi_2,\ldots,\pi_k:[n]\rightarrow[n]$ used in Definition~\ref{def:chy} to define $U_n$, we only need to know the permutations $\pi_i$ and $\pi_j$. To generate these, we toss in i.i.d. points $i_1,i_2,\ldots,i_n,j_1,j_2,\ldots,j_n$ into the unit interval $I$; the Euclidean order of the points $i_1,i_2,\ldots,i_n$ naturally defines $\pi_i$ and similarly the points $j_1,j_2,\ldots,j_n$ naturally define $\pi_j$.\footnote{The exception being when some of the points $i_1,i_2,\ldots,i_n$ or of the points $j_1,j_2,\ldots,j_n$ coincide, in which case the order of these points does not determine a permutation. This event however happens almost never.} So, we can view $Y_n$ as a random variable on the probability space $I^{2n}$. Observe that if $\mathfrak{s}=(i_1,i_2,\ldots,i_n,j_1,j_2,\ldots,j_n)$ and $\mathfrak{s}'=(i'_1,i'_2,\ldots,i'_n,j'_1,j'_2,\ldots,j'_n)$ are two elements of $I^{2n}$ that differ in only one coordinate, then $|Y_n(\mathfrak{s})-Y_n(\mathfrak{s}')|\le \frac{2}{n}$. Thus the Method of Bounded Differences (see~\cite{McDiarmid1989}) tells us that 
	$$\PROB\left[E_n\right]\le\PROB\left[|Y_n-\EXP[Y_n]|>\sqrt[4]{1/n}\right]\le 2\exp\left(-\frac{2(\sqrt[4]{1/n})^2}{2n\cdot\left(\frac{2}{n}\right)^2}\right)=2\exp\left(-\sqrt{n}/4\right)\;.$$
	
	Because the sequence $\left(2\exp\left(-\sqrt{n}/4\right)\right)_{n=1}^\infty$ is summable, the Borel--Cantelli lemma allows to conclude that only finitely many events $E_n$ occur, almost surely. Thus, almost surely, for any weak$^*$ accumulation point $U$ of the sequence $U_1,U_2,U_3,\ldots$, we have 
	\begin{equation}\label{eq:reproduktor}
		\iint_{[p_1,p_2]\times[r_1,r_2]}U\d(\lambda^{\oplus2})=c(p_2-p_1)(r_2-r_1)\;.
	\end{equation}
	By applying the union bound, we obtain that~\eqref{eq:reproduktor} holds for all (countably many) choices of $i,j,p_1,p_2,r_1,r_2$, almost surely. Since the elements of $\mathcal{I}$ are intervals, the above system of rectangles $[p_1,p_2]\times[r_1,r_2]$ generates the Borel $\sigma$-algebra on $I^2$. Consequently, we obtain that $U\equiv W^{\Join\mathcal I}$, almost surely.
	
	The ``moreover'' part obviously follows from the proof.
\end{proof}

\subsection{Jensen's inequality and steppings}

Recall that one of the possible formulations of Jensen's inequality says that if $(\Omega,\lambda)$ is a measurable space with $\lambda(\Omega)>0$, $g:\Omega\rightarrow\mathbb R$ is a measurable function and $f:\mathbb R\rightarrow\mathbb R$ is a concave function then \begin{equation}\label{eq:Jensen}
f\left(\frac 1{\lambda(\Omega)}\int_\Omega g(x)\right)\ge\frac 1{\lambda(\Omega)}\int_\Omega f(g(x))\;.
\end{equation}
We use this formulation of Jensen's inequality to prove the following simple lemma.

\begin{lem}\label{lem:concavetriv}
	Let $f\colon[0,1]\rightarrow\mathbb R$ be a continuous and strictly concave function. Let $\Gamma\colon I^2\rightarrow [0,1]$ be a step-graphon with steps $I_1,I_2,\ldots, I_k$, and let $W\colon I^2\rightarrow [0,1]$ be another graphon such that $\int_{I_i\times I_j}W=\int_{I_i\times I_j}\Gamma$ for every $i,j=1,2,\ldots,k$. Then $\INT_f(W)\le\INT_f(\Gamma)$.
\end{lem}

\begin{proof}
	It clearly suffices to show that for every $i,j=1,2,\ldots,k$ it holds
	\begin{equation}\nonumber
	\int_{I_i}\int_{I_j}f(W(x,y))\le\int_{I_i}\int_{I_j}f(\Gamma(x,y))\;.
	\end{equation}
	So let us fix  $i,j$, and let $C_{i,j}$ be the constant for which $\Gamma_{|I_i\times I_j}=C_{i,j}$ almost everywhere. Then we have
	\begin{equation}\nonumber
	\begin{aligned}
	\int_{I_i}\int_{I_j}f(W(x,y))&\stackrel{(\ref{eq:Jensen})}{\le}\lambda^{\oplus 2}(I_i\times I_j)\cdot f\left(\frac 1{\lambda^{\oplus 2}(I_i\times I_j)}\int_{I_i}\int_{I_j}W(x,y)\right)\\
	&=\lambda^{\oplus 2}(I_i\times I_j)\cdot f\left(C_{i,j}\right)\\
	&=\int_{I_i}\int_{I_j}f(\Gamma(x,y))\;,
	\end{aligned}
	\end{equation}
	as we wanted.
\end{proof}

\section{Summaries of proofs}
In this section, we give an overview of the proof of Theorem~\ref{thm:compactANDminimizer}\ref{en:improve} in Section~\ref{ssec:overviewPartImprove}. Then, we explain in Section~\ref{ssec:connectionwithRL} that this proof can be viewed as an infinitesimal counterpart to the index-pumping lemma. Last, in Section~\ref{ssec:overviewPartCompact} we give a detailed outline of Theorem~\ref{thm:compactANDminimizer}\ref{en:compact}.
\subsection{Overview of proof of  Theorem~\ref{thm:compactANDminimizer}\ref{en:improve}}\label{ssec:overviewPartImprove}

Suppose for simplicity that the sequence $\Gamma_1,\Gamma_2,\Gamma_3,\ldots$ converges to $W$ in the weak$^*$ topology. The key step to the proof of Theorem~\ref{thm:compactANDminimizer}$\ref{en:improve}$ is Lemma~\ref{lem:neostraNerovnost}. There we prove that whenever we fix a sequence $(B_n)_{n=1}^{\infty}$ of measurable subsets of $I$ and define a new version $\Gamma'_n$ of $\Gamma_n$ (for every $n$) by ``shifting the set $B_n$ to the left side of the interval $I$'', then any weak$^*$ accumulation point $\widetilde W$ of the sequence $\Gamma'_1,\Gamma'_2,\Gamma'_3,\ldots$ satisfies $\INT_f(\widetilde W)\le\INT_f(W)$.
As this result relies on Jensen's inequality, we actually get $\INT_f(\widetilde W)<\INT_f(W)$ when we choose the sets $B_n$ carefully. ``Carefully'' means that each of the integrals $\int_{B_n}\int_{B_n}\Gamma_n(x,y)$ differs from the integral $\int_{B_n}\int_{B_n}W(x,y)$ at least by some given $\varepsilon>0$.
But observe that if the graphon $W$ is not a cut-norm accumulation point of the sequence $\Gamma_1,\Gamma_2,\Gamma_3,\ldots$ then it is always possible to choose the sets $B_n$.

\subsection{Connection between the proof of  Theorem~\ref{thm:compactANDminimizer}\ref{en:improve} and proofs of regularity lemmas}
\label{ssec:connectionwithRL}
Graphons could be regarded as ``the ultimate regularization''. Thus, it is instructive to see how our proof relates to the usual proofs of regularity lemmas (of which the weak regularity lemma of Frieze and Kannan~\cite{Frieze1999} is the most relevant). Recall that in these proofs of regularity lemmas one keeps refining a partition of a graph until the partition is regular. 

Let us give details. Let $f:[0,1]\rightarrow \mathbb R$ be an arbitrary continuous and strictly concave function. Suppose that $G$ is an $n$-vertex graph, and let $\mathcal P=(P_i)_{i=1}^k$ be a partition of $V(G)$ into sets. Then for each $i,j\in[k]$, we define $d_{ij}:=\frac{\sum_{u\in P_i, v\in P_j}\mathbf{1}_{uv\in E(G)}}{|P_i|\cdot|P_j|}$ (with the convention $\frac{0}{0}=0$). If $i\neq j$ then $d_{ij}$ corresponds to the bipartite density of the pair $G[P_i,P_j]$, and otherwise this corresponds to the density of the graph $G[P_i]$. Then we write
$$\INT_f(G;\mathcal P):=\sum_{i=1}^k\sum_{j=1}^k \frac{|P_i|\cdot |P_j|}{n^2}\cdot f(d_{ij})\;.$$
Note that we can express $\INT_f(G;\mathcal P)$ as $\INT_f(W_{G;\mathcal P})$, where $W_{G;\mathcal P}$ is a graphon representation of densities of $G$ according to the partition $\mathcal P$.
The index-pumping lemma, which we state here in the setting of the weak regularity lemma, asserts that non-regular partitions can  be refined while controlling the index. Let us recall that a partition $\mathcal P=(P_i)_{i=1}^k$ of $V(G)$ is \emph{weak $\epsilon$-regular} if for each $B\subset V(G)$ we have
$$e(G[B])=\frac{1}{2}\sum_{i=1}^k\sum_{j=1}^k d_{i,j}|B\cap P_i|\cdot |B\cap P_j|\pm \epsilon n^2\;.$$
\begin{lem}[Index-pumping lemma]\label{lem:indexpumping}
	Suppose that $\mathcal C$ is a partition of a graph $G$. If $B\subset V(G)$ is a witness that $\mathcal C$ is not $\epsilon$-regular, then splitting each cell $C\in \mathcal C$ into $C\cap B$ and $C\setminus B$ yields a partition $\mathcal D$ for which $$\INT_{x\mapsto -x^2}(G;\mathcal D)<\INT_{x\mapsto -x^2}(G;\mathcal C)-\frac{\epsilon^2}4\;.$$
\end{lem}
For completeness, let us recall the the proof of the weak regularity lemma, which states that for each $\epsilon>0$, each graph has a weak $\epsilon$-regular partition with at most $2^{\lceil\frac4{\epsilon^2}\rceil}$ parts. One starts with a singleton partition. At any stage, if the current partition is not weak $\epsilon$-regular, then Lemma~\ref{lem:indexpumping} allows to decrease the index by at least $\frac{\epsilon^2}4$ while doubling the number of cells in the partition. Since $\INT_{x\mapsto -x^2}(G;\cdot )\in[-1,0]$, we must terminate in at most $\lceil\frac4{\epsilon^2}\rceil$ steps.

Let us now draw the analogy between Lemma~\ref{lem:indexpumping} and Theorem~\ref{thm:compactANDminimizer}\ref{en:improve} and its proof. Firstly, note that Theorem~\ref{thm:compactANDminimizer} allows other functions than $x\mapsto -x^2$ used in Lemma~\ref{lem:indexpumping}. This is however not a serious restriction. Indeed, replacing the so-called ``defect form of 
the Cauchy--Schwarz inequality'' in the usual proof of Lemma~\ref{lem:indexpumping} by Jensen's inequality, we could obtain a statement for general continuous strictly concave functions. In fact, this has already been used in~\cite{Scott}, where a different choice of a concave function was necessary. So let us now move to the main analogy. Let us consider $W$ as in Theorem~\ref{thm:compactANDminimizer}\ref{en:improve}. As in Section~\ref{ssec:overviewPartImprove}, let us assume that $W$ is the weak* limit of $\Gamma_1,\Gamma_2,\Gamma_3\ldots$. Suppose that $W$ is not an accumulation point of $\Gamma_{1},\Gamma_{2},\Gamma_{3},\ldots$ with respect to the cut-norm, and let $B_1,B_2,B_3,\ldots\subset I$ be witnesses for this. Now, the ``shifting $B_n$ to the left'' described in Section~\ref{ssec:overviewPartImprove} can be viewed as splitting each interval $J\subset I$ into $J\cap B_n$ and $J\setminus B_n$ (we think of $J$ as being very small, thus representing an ``infinitesimally small cluster''), just as in Lemma~\ref{lem:indexpumping}.

We pose a conjecture which goes in this direction in Section~\ref{ssec:ConjectureFiniteGraphs}.

\subsection{Overview of proof of Theorem~\ref{thm:compactANDminimizer}$\ref{en:compact}$}\label{ssec:overviewPartCompact}
Let us begin with the most straightforward attempt for a proof. For now, let us work with the simplifying assumption that all accumulation points are actually limits. As we shall see later, this simplifying assumption is a major cheat for which an extra patch will be needed. Then, let
\begin{equation*}
m:=\inf\left\{\INT_f(W)\colon W\in\LIM(\Gamma_1,\Gamma_2,\Gamma_3,\ldots)\right\}\;.
\end{equation*}
For each $k\in\mathbb{N}$, let us fix a sequence $\Gamma_1^k,\Gamma_2^k,\Gamma_3^k,\ldots$ of versions of $\Gamma_1,\Gamma_2,\Gamma_3,\ldots$ which converges in the weak$^*$ topology to a graphon $\widetilde{W}_k$ with $\INT_f(\widetilde{W}_k)<m+\frac1k$. Now, we might diagonalize and hope that any weak$^*$ accumulation point (whose existence is guaranteed by the Banach--Alaoglu Theorem) $W^*$ of the sequence $\Gamma_1^1,\Gamma_2^2,\Gamma_3^3,\ldots$ satisfies $\INT_f(W^*)\le m$. The reason for this hope being vain is the discontinuity of $\INT_f(\cdot)$ with respect to the weak$^*$ topology. 
\begin{figure}
	\includegraphics[scale=1]{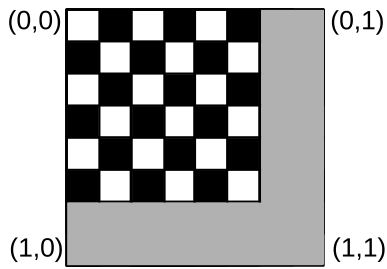}
	\caption{The graphon $\widetilde{W}_3$ from Section~\ref{ssec:overviewPartCompact}. Value 0 is white, value $\frac12$ is gray, value $1$ is black.}
	\label{fig:almostchess}
\end{figure}
As an example, let us take a situation when each $\widetilde{W}_k$ is a $2(k+2)\times 2(k+2)$-chessboard $\{0,1\}$-valued function, with the last two rows and columns having value $\frac12$ (see Figure~\ref{fig:almostchess}). In other words, most of each graphon $\widetilde{W}_k$ corresponds to a complete balanced bipartite graphon, to which an additional artificial subdivision to each of its parts to $k$ subparts was introduced. These subparts were interlaced one after another, except that the vertices of the last subpart of each part were mixed together. (These graphons were clearly chosen nonoptimally in the sense that the mixing of the last two parts is undesired. We chose these graphons in this example here to have richer features to study.) All the graphons $\widetilde{W}_k$ have small values of $\INT_f(\cdot)$. On the other hand, the weak$^*$ limit of the sequence is the graphon $W_{\mathrm{const}}\equiv\frac12$ whose value $\INT_f(\cdot)$ is bigger. There is a lesson to learn from this example. While for larger $k$, the versions in the sequence $\Gamma_1^k,\Gamma_2^k,\Gamma_3^k,\ldots$ will be aligned on $I$ in a more optimal way locally, the global structure may get undesirably more convoluted as $k\rightarrow \infty$. To remedy this, we consider a sequence of version of $\Gamma_1,\Gamma_2,\Gamma_3,\ldots$ in which the structure of measure-preserving transformation on a rough level is inherited from measure preserving transformations leading to $\widetilde{W}_1$. Within each step corresponding to the step-graphon $\widetilde{W}_1$, the structure of the measure-preserving transformation is inherited from measure preserving transformations leading to $\widetilde{W}_2$, and so on. An example of this procedure is given in Figure~\ref{fig:twosteps}. It can be shown that any weak$^*$ accumulation point $W^*$ of these reordered graphons has the property that $\INT_f(W^*)\le \limsup_n \INT_f\widetilde{W}_n$, as was needed.
\begin{figure}
	\includegraphics[scale=.85]{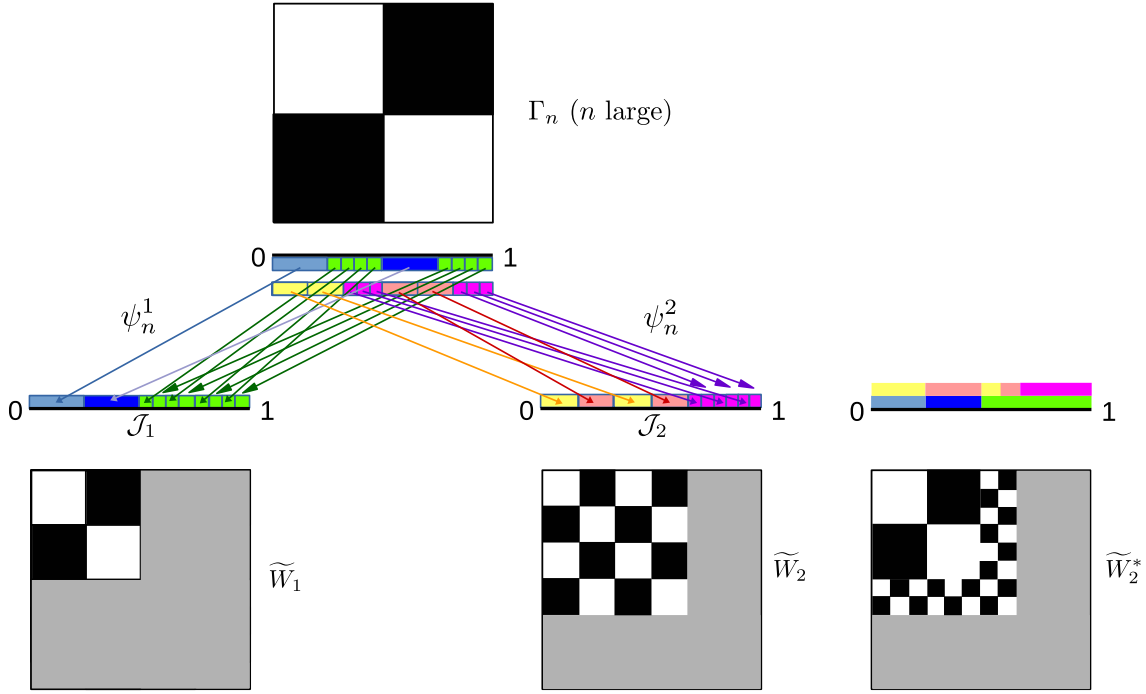}
	\caption{An example of reordering from Section~\ref{ssec:overviewPartCompact}. The top shows a graphon $\Gamma_n$, versions of which are close to $\widetilde{W}_1$ and $\widetilde{W}_2$ in the weak$^*$ topology. The two measure preserving transformations $\psi^{1}_n$ and $\psi^{2}_n$ which witness this closeness are shown with colors. The graphon $\widetilde{W}_2^*$ emerges by taking the partition whose global structure from $\widetilde{W}_1$ is refined according to the more local structure from $\widetilde{W}_2$. Iterating this process would lead to a sequence of graphons $(\widetilde{W}_k^*)_k$ which has the property that for any weak$^*$ accomulation point $W^*$ we have $\INT_f(W^*)\le \limsup_n \INT_f\widetilde{W}_n$.}
	\label{fig:twosteps}
\end{figure}

Let us now explain why the assumption that all sequences converge weak$^*$ leaves a substantial gap in the proof. Recall that the information how the partition $\mathcal J^k$ of $U_{k}$ interacts with the measure preserving almost-bijections on graphons $\mathfrak{s}_k\subset (\Gamma_1,\Gamma_2,\Gamma_3,\ldots)$ that converge to $\widetilde{W}_k$ gives us crucial directions as how to reorder and refine the subsequence of graphons $\mathfrak{s}_{k+1}$ that converges to $\widetilde{W}_{k+1}$. Let us again stress that while the existence of the subsequences $\mathfrak{s}_{j}$ is guaranteed by weak$^*$ compactness, we have no control on their properties. So, it can be that $\mathfrak{s}_{k}$ is disjoint from $\mathfrak{s}_{k+1}$. In other words, we do not get the needed information how to reorder and refine the graphons in $\mathfrak{s}_{k+1}$. To remedy this problem, we prove a lemma (Lemma~\ref{lem:ACC=LIM}) which says that for every sequence $\Gamma_1,\Gamma_2,\Gamma_3,\ldots$ of graphons there exists a subsequence $\Gamma_{k_1},\Gamma_{k_2},\Gamma_{k_3},\ldots$ such that
\begin{equation*}
\inf\left\{\INT_f(W)\colon W\in\ACC(\Gamma_{k_1},\Gamma_{k_2},\Gamma_{k_3},\ldots)\right\}=\inf\left\{\INT_f(W)\colon W\in\LIM(\Gamma_{k_1},\Gamma_{k_2},\Gamma_{k_3},\ldots)\right\}\;.
\end{equation*}
Applying this lemma first, the arguments above become sound for the subsequence $\Gamma_{k_1},\Gamma_{k_2},\Gamma_{k_3},\ldots$.

\section{Proof of Theorem~\ref{thm:compactANDminimizer}$\ref{en:improve}$}\label{sec:proofImprove}

The following key lemma (or its subsequent corollary) is used in both proofs of Theorem~\ref{thm:compactANDminimizer}\ref{en:improve} and Theorem~\ref{thm:compactANDminimizer}\ref{en:compact}.

\begin{lem}\label{lem:neostraNerovnost}
Suppose that $f:[0,1]\rightarrow\mathbb R$ is an arbitrary continuous and strictly concave function. Suppose that $\Gamma_1,\Gamma_2,\Gamma_3,\ldots:I^2\rightarrow[0,1]$ is a sequence of graphons which converges to a graphon $W:I^2\rightarrow[0,1]$ in the weak$^*$ topology. Suppose that $B_1,B_2,B_3,\ldots$ is an arbitrary sequence of subsets of $I$. 
For each $n$, let $\mathcal J_n$ be the ordered partition of $I$ into two sets $B_n$ and $I\setminus B_n$ (in this order).
Then every graphon $\widetilde W$ that is a weak* accumulation point of the sequence $\prescript{}{\mathcal J_1}\Gamma_1,\prescript{}{\mathcal J_2}\Gamma_2,\prescript{}{\mathcal J_3}\Gamma_3,\ldots$ satisfies $\INT_f(\widetilde W)\le\INT_f(W)$.

Moreover, suppose that for the sequence $n_1<n_2<n_3<\ldots$ for which $\prescript{}{\mathcal J_{n_1}}\Gamma_{n_1},\prescript{}{\mathcal J_{n_2}}\Gamma_{n_2},\prescript{}{\mathcal J_{n_3}}\Gamma_{n_3},\ldots$ weak* converges to $\widetilde W$, we have that ${\bf 1}_{B_{n_1}},{\bf 1}_{B_{n_2}},{\bf 1}_{B_{n_3}},\ldots$ converges to a function $\psi\colon I\rightarrow[0,1]$ in the weak$^*$ topology. Let $\theta\colon I\rightarrow I$ be defined by $\theta(x)=\int_0^x\psi(y)\:\d(y)$.
If we have 
\begin{equation}\label{eq:TYR}
\lambda^{\otimes 2}\left(\left\{(x,y)\in I^2\::\: \psi(x)>0, \psi(y)>0, W(x,y)\neq\widetilde W(\theta(x),\theta(y))\right\}\right)>0 
\end{equation}
then $\INT_f(\widetilde W)<\INT_f(W)$.
\end{lem}

\begin{proof}
By passing to a subsequence, we may assume that the sequence $\prescript{}{\mathcal J_1}\Gamma_1,\prescript{}{\mathcal J_2}\Gamma_2,\prescript{}{\mathcal J_3}\Gamma_3,\ldots$ is convergent to $\widetilde W$ in the weak$^*$ topology, and that the sequence ${\bf 1}_{B_1},{\bf 1}_{B_2},{\bf 1}_{B_3},\ldots$ converges in the weak$^*$ topology to $\psi\colon I\rightarrow[0,1]$.
We define $\xi\colon I\rightarrow I$ by $\xi(x)=\theta(1)+\int_0^x(1-\psi(y))\:\d(y)$.

\begin{claim}\label{cl:ManuelSpi}
	For every two intervals $[p_1,p_2],[q_1,q_2]\subset I$  we have 
	\begin{equation}\label{eq:mozart}
	\begin{aligned}
	\int_{p_1}^{p_2}\int_{q_1}^{q_2}W(x,y)&=\int_{p_1}^{p_2}\int_{q_1}^{q_2}\widetilde W(\theta(x),\theta(y))\psi(x)\psi(y)\\
	&\quad+\int_{p_1}^{p_2}\int_{q_1}^{q_2}\widetilde W(\theta(x),\xi(y))\psi(x)(1-\psi(y))\\
	&\quad+\int_{p_1}^{p_2}\int_{q_1}^{q_2}\widetilde W(\xi(x),\theta(y))(1-\psi(x))\psi(y)\\
	&\quad+\int_{p_1}^{p_2}\int_{q_1}^{q_2}\widetilde W(\xi(x),\xi(y))(1-\psi(x))(1-\psi(y))\;.
	\end{aligned}
	\end{equation}
\end{claim}
\begin{proof}[Proof of Claim~\ref{cl:ManuelSpi}]
	By using the fact that $\Gamma_n\stackrel{w^*}{\rightarrow}W$ together with the identity $ab+a(1-b)+(1-a)b+(1-a)(1-b)=1$ we get that
	\begin{equation}\label{eq:roztrzeniNaCtyriIntegraly}
	\begin{aligned}	\int_{p_1}^{p_2}\int_{q_1}^{q_2}W(x,y)&=\lim\limits_{n\rightarrow\infty}\int_{p_1}^{p_2}\int_{q_1}^{q_2}\Gamma_n(x,y)\\
	&=
	\lim_{n\rightarrow\infty}\int_{p_1}^{p_2}\int_{q_1}^{q_2}\Gamma_n(x,y){\bf 1}_{{B_n}}(x){\bf 1}_{{B_n}}(y)\\
	&\qquad+\lim_{n\rightarrow\infty}\int_{p_1}^{p_2}\int_{q_1}^{q_2}\Gamma_n(x,y){\bf 1}_{{B_n}}(x)(1-{\bf 1}_{{B_n}}(y))\\
	&\qquad+\lim_{n\rightarrow\infty}\int_{p_1}^{p_2}\int_{q_1}^{q_2}\Gamma_n(x,y)(1-{\bf 1}_{{B_n}}(x)){\bf 1}_{{B_n}}(y)\\
	&\qquad+\lim_{n\rightarrow\infty}\int_{p_1}^{p_2}\int_{q_1}^{q_2}\Gamma_n(x,y)(1-{\bf 1}_{{B_n}}(x))(1-{\bf 1}_{{B_n}}(y))\;.
	\end{aligned}
	\end{equation}
	Next we rewrite the integral following the first limit on the right-hand side of (\ref{eq:roztrzeniNaCtyriIntegraly}). To this end, we use the notation from~\eqref{eq:pocitac} together with the obvious differentiation formula
	\begin{equation}\label{eq:derivace}
	(\alpha_{\mathcal J_n,1})'(x)={\bf 1}_{B_n}(x)\quad\text{for a.e. }x\in I
	\end{equation}
	(and also, we use the fact that $\alpha_{\mathcal J_n,1|B_n}$ is an almost-bijection from $B_n$ onto the interval $[0,\int_0^1{\bf 1}_{B_n}(y)]$, and so it makes sense to talk about its inverse).
	We have
	\begin{equation}\label{eq:substituce}
	\begin{aligned}
	&\int_{p_1}^{p_2}\int_{q_1}^{q_2}\Gamma_n(x,y){\bf 1}_{{B_n}}(x){\bf 1}_{{B_n}}(y)\\
	\JUSTIFY{integration by substitution}=&\int_{\alpha_{\mathcal J_n,1}(p_1)}^{\alpha_{\mathcal J_n,1}(p_2)}\int_{\alpha_{\mathcal J_n,1}(q_1)}^{\alpha_{\mathcal J_n,1}(q_2)}\Gamma_n(\alpha_{\mathcal J_n,1}^{-1}(x),\alpha_{\mathcal J_n,1}^{-1}(y))\\
	\JUSTIFY{$\gamma_{\mathcal J_n}(x)=\alpha_{\mathcal J_n,1}(x)$ for every $x\in B_n$}=&\int_{\alpha_{\mathcal J_n,1}(p_1)}^{\alpha_{\mathcal J_n,1}(p_2)}\int_{\alpha_{\mathcal J_n,1}(q_1)}^{\alpha_{\mathcal J_n,1}(q_2)}\Gamma_n(\gamma_{\mathcal J_n}^{-1}(x),\gamma_{\mathcal J_n}^{-1}(y))\\
	=&\int_{\alpha_{\mathcal J_n,1}(p_1)}^{\alpha_{\mathcal J_n,1}(p_2)}\int_{\alpha_{\mathcal J_n,1}(q_1)}^{\alpha_{\mathcal J_n,1}(q_2)}\prescript{}{\mathcal J_n}\Gamma_n(x,y)\;.
	\end{aligned}
	\end{equation}
	Therefore, we have
	\begin{equation}\label{eq:error}
	\begin{aligned}
	&\left|\int_{p_1}^{p_2}\int_{q_1}^{q_2}\Gamma_n(x,y){\bf 1}_{{B_n}}(x){\bf 1}_{{B_n}}(y)-\int_{\theta(p_1)}^{\theta(p_2)}\int_{\theta(q_1)}^{\theta(q_2)}\prescript{}{\mathcal J_n}\Gamma_n(x,y)\right|\\
	\stackrel{(\ref{eq:substituce})}{=}&\left|\int_{\alpha_{\mathcal J_n,1}(p_1)}^{\alpha_{\mathcal J_n,1}(p_2)}\int_{\alpha_{\mathcal J_n,1}(q_1)}^{\alpha_{\mathcal J_n,1}(q_2)}\prescript{}{\mathcal J_n}\Gamma_n(x,y)-\int_{\theta(p_1)}^{\theta(p_2)}\int_{\theta(q_1)}^{\theta(q_2)}\prescript{}{\mathcal J_n}\Gamma_n(x,y)\right|\\
	\le&\left|\alpha_{\mathcal J_n,1}(p_1)-\theta(p_1)\right|+\left|\alpha_{\mathcal J_n,1}(p_2)-\theta(p_2)\right|+\left|\alpha_{\mathcal J_n,1}(q_1)-\theta(q_1)\right|+\left|\alpha_{\mathcal J_n,1}(q_2)-\theta(q_2)\right|\;.
	\end{aligned}
	\end{equation}
	The fact that ${\bf 1}_{B_n}\stackrel{w^*}{\rightarrow}\psi$ immediately implies that $\alpha_{\mathcal J_n,1}(x)\rightarrow \theta(x)$ for every $x\in I$, and so we conclude that the right-hand side, and thus also the left-hand side, of (\ref{eq:error}), tends to 0. Therefore (note that the following limits exist as $\Gamma_n\stackrel{w^*}{\rightarrow}W$)
	\begin{equation}\label{eq:sumand1}
	\begin{aligned}
	\lim_{n\rightarrow\infty}\int_{p_1}^{p_2}\int_{q_1}^{q_2}\Gamma_n(x,y){\bf 1}_{{B_n}}(x){\bf 1}_{{B_n}}(y)&=
	\lim_{n\rightarrow\infty}\int_{\theta(p_1)}^{\theta(p_2)}\int_{\theta(q_1)}^{\theta(q_2)}\prescript{}{\mathcal J_n}\Gamma_n(x,y)\\
	\JUSTIFY{$\prescript{}{\mathcal J_n}\Gamma_n\stackrel{w^*}{\rightarrow}\widetilde W$}&=\int_{\theta(p_1)}^{\theta(p_2)}\int_{\theta(q_1)}^{\theta(q_2)}\widetilde W(x,y)\\
	\JUSTIFY{integration by substitution}&=\int_{p_1}^{p_2}\int_{q_1}^{q_2}\widetilde W(\theta(x),\theta(y))\psi(x)\psi(y)\;.
	\end{aligned}
	\end{equation}
	In a very analogous way as we derived (\ref{eq:sumand1}), one can verify that 
	\begin{align}\label{eq:sumand2}
		\lim_{n\rightarrow\infty}\int_{p_1}^{p_2}\int_{q_1}^{q_2}\Gamma_n(x,y){\bf 1}_{{B_n}}(x)(1-{\bf 1}_{{B_n}}(y))&=\int_{p_1}^{p_2}\int_{q_1}^{q_2}\widetilde W(\theta(x),\xi(y))\psi(x)(1-\psi(y))\;,
		\\
		\label{eq:sumand3}
		\lim_{n\rightarrow\infty}\int_{p_1}^{p_2}\int_{q_1}^{q_2}\Gamma_n(x,y)(1-{\bf 1}_{{B_n}}(x)){\bf 1}_{{B_n}}(y)&=\int_{p_1}^{p_2}\int_{q_1}^{q_2}\widetilde W(\xi(x),\theta(y))(1-\psi(x))\psi(y)
		\;,
		\\
		\label{eq:sumand4}
		\lim_{n\rightarrow\infty}\int_{p_1}^{p_2}\int_{q_1}^{q_2}\Gamma_n(x,y)(1-{\bf 1}_{{B_n}}(x))(1-{\bf 1}_{{B_n}}(y))&=\int_{p_1}^{p_2}\int_{q_1}^{q_2}\widetilde W(\xi(x),\xi(y))(1-\psi(x))(1-\psi(y))\;.
	\end{align}
	By putting (\ref{eq:roztrzeniNaCtyriIntegraly}), (\ref{eq:sumand1}), (\ref{eq:sumand2}), (\ref{eq:sumand3}) and (\ref{eq:sumand4}) together, we get~\eqref{eq:mozart}.
\end{proof}

Since the sets of the form $[p_1,p_2]\times [q_1,q_2]$ generate the Borel $\sigma$-algebra on $I^2$, we conclude from Claim~\ref{cl:ManuelSpi} that for almost every $(x,y)\in I^2$ we have that
\begin{equation}\label{eq:convexCombination}
\begin{aligned}
W(x,y)=&\widetilde W(\theta(x),\theta(y))\psi(x)\psi(y)+\widetilde W(\theta(x),\xi(y))\psi(x)(1-\psi(y))\\
&\quad+\widetilde W(\xi(x),\theta(y))(1-\psi(x))\psi(y)+\widetilde W(\xi(x),\xi(y))(1-\psi(x))(1-\psi(y))\;.
\end{aligned}
\end{equation}
Note that the right-hand side of (\ref{eq:convexCombination}) is a convex combination of the four terms
\begin{equation}\label{def:termsFromConvComb}
\widetilde W(\theta(x),\theta(y))\;,\quad\widetilde W(\theta(x),\xi(y))\;,\quad\widetilde W(\xi(x),\theta(y))\;,\quad\widetilde W(\xi(x),\xi(y))\;.
\end{equation}
Therefore we have
\begin{equation}\label{eq:ostraVsNeostra}
\begin{aligned}
\INT_f(W)&=\int_0^1\int_0^1f(W(x,y))\\
\JUSTIFY{$f$ is concave}&\stackrel{(\ref{eq:convexCombination})}{\ge}\int_0^1\int_0^1f\left(\widetilde W(\theta(x),\theta(y))\right)\psi(x)\psi(y)\\
&\qquad+\int_0^1\int_0^1f\left(\widetilde W(\theta(x),\xi(y))\right)\psi(x)(1-\psi(y))\\
&\qquad+\int_0^1\int_0^1f\left(\widetilde W(\xi(x),\theta(y))\right)(1-\psi(x))\psi(y)\\
&\qquad+\int_0^1\int_0^1f\left(\widetilde W(\xi(x),\xi(y))\right)(1-\psi(x))(1-\psi(y))\\
\JUSTIFY{integration by substitution}&=\int_0^{\theta(1)}\int_0^{\theta(1)}f\left(\widetilde W(x,y)\right)+\int_0^{\theta(1)}\int_{\theta(1)}^1f\left(\widetilde W(x,y)\right)\\
&\qquad+\int_{\theta(1)}^1\int_0^{\theta(1)}f\left(\widetilde W(x,y)\right)+\int_{\theta(1)}^1\int_{\theta(1)}^1f\left(\widetilde W(x,y)\right)\\
&=\int_0^1\int_0^1f\left(\widetilde W(x,y)\right)=\INT_f(\widetilde W)\;.
\end{aligned}
\end{equation}

To prove the ``moreover'' part, suppose that we have~\eqref{eq:TYR}.
Then the convex combination (\ref{eq:convexCombination}) is not trivial on a set of positive measure. This is all we need as then we have a sharp inequality in (\ref{eq:ostraVsNeostra}) because $f$ is strictly concave.
\end{proof}

We do not use the next corollary right now but we will need it in Section~\ref{sec:proofCompact}.

\begin{cor}\label{cor:zmensiEntropii}
Suppose that $f:[0,1]\rightarrow\mathbb R$ is an arbitrary continuous and strictly concave function. Suppose that $\Gamma_1,\Gamma_2,\Gamma_3,\ldots:I^2\rightarrow[0,1]$ is a sequence of graphons which converges to a graphon $W:I^2\rightarrow[0,1]$ in the weak$^*$ topology. Suppose that $\ell$ is a fixed natural number and that for every $n$, $\mathcal J_n$ is an ordered partition of $I$ into $\ell$ sets $B_1^n,B_2^n,\ldots,B_{\ell}^n$.
Then for every graphon $\widetilde W$ that is a weak* accumulation point of the graphons $\prescript{}{\mathcal J_1}\Gamma_1,\prescript{}{\mathcal J_2}\Gamma_2,\prescript{}{\mathcal J_3}\Gamma_3,\ldots$ we have $\INT_f(\widetilde W)\le\INT_f(W)$.
\end{cor}

\begin{proof}
For every natural number $n$ and every $i\in\{1,\ldots,\ell\}$, we denote by $\mathcal J_n^i$ the ordered partition of $I$ consisting of the sets $B_{\ell-i+1}^n,B_{\ell-i+2}^n,\ldots,B_{\ell}^n$ and $I\setminus\bigcup_{j=\ell-i+1}^{\ell} B_{j}^n$ (in this order).
Consider these $\ell+1$ sequences of graphons:
\begin{equation}\nonumber
\begin{aligned}
\mathcal S_0&\colon \Gamma_1,\Gamma_2,\Gamma_3,\ldots\\
\mathcal S_1&\colon \prescript{}{\mathcal J_1^1}\Gamma_1,\prescript{}{\mathcal J_2^1}\Gamma_2,\prescript{}{\mathcal J_3^1}\Gamma_3,\ldots\\
\mathcal S_2&\colon \prescript{}{\mathcal J_1^2}\Gamma_1,\prescript{}{\mathcal J_2^2}\Gamma_2,\prescript{}{\mathcal J_3^2}\Gamma_3,\ldots\\
\vdots\\
\mathcal S_{\ell}&\colon \prescript{}{\mathcal J_1^{\ell}}\Gamma_1,\prescript{}{\mathcal J_2^{\ell}}\Gamma_2,\prescript{}{\mathcal J_3^{\ell}}\Gamma_3,\ldots\;,\\
\end{aligned}
\end{equation}
so that the sequence $\mathcal S_{\ell}$ is precisely $\prescript{}{\mathcal J_1}\Gamma_1,\prescript{}{\mathcal J_2}\Gamma_2,\prescript{}{\mathcal J_3}\Gamma_3,\ldots$.
Let us fix $\widetilde W\in\ACC(\mathcal S_{\ell})$.
By passing to a subsequence, we may assume that the sequence $\mathcal S_i$ converges to some graphon $W_i$ in the weak$^*$ topology for every $i=1,2,\ldots,\ell-1$.
It remains to apply Lemma~\ref{lem:neostraNerovnost} $\ell$-times in a row. First, we apply it on the sequence $\mathcal S_0$ of graphons and on the sequence $B_{\ell}^1,B_{\ell}^2,B_{\ell}^3,\ldots$ of subsets of $I$ to conclude that $\INT_f(W_1)\le\INT_f(W)$. Next, we apply it on the sequence $\mathcal S_1$ of graphons and on the sequence $B_{\ell-1}^1,B_{\ell-1}^2,B_{\ell-1}^3,\ldots$ of subsets of $I$ to conclude that $\INT_f(W_2)\le\INT_f(W_1)\le\INT_f(W)$. In the last step, we apply it on the sequence $\mathcal S_{\ell-1}$ of graphons and on the sequence $B_1^1,B_1^2,B_1^3,\ldots$ of subsets of $I$ to conclude that $\INT_f(\widetilde W)\le\INT_f(W_{\ell-1})\le\ldots\le\INT_f(W_1)\le\INT_f(W)$.
\end{proof}

Now we can prove Theorem~\ref{thm:compactANDminimizer}\ref{en:improve}.

By passing to a subsequence, we may assume that the sequence $\Gamma_1,\Gamma_2,\Gamma_3,\ldots$ converges to $W$ in the weak$^*$ topology. As $W$ is not an accumulation point of the sequence  $\Gamma_1,\Gamma_2,\Gamma_3,\ldots$ in the cut-norm, there is $\varepsilon>0$ and a natural number $n_0$ such that $\|\Gamma_n-W\|_\square\ge\varepsilon$ for every $n\ge n_0$. By passing to a subsequence, we may suppose that $\|\Gamma_n-W\|_\square\ge\varepsilon$ for every natural number $n$. By the definition of the cut-norm, there is a sequence $B_1,B_2,B_3,\ldots$ of subsets of $I$ such that for every natural number $n$ we have $\left|\int_{x\in B_n}\int_{y\in B_n} \left(\Gamma_n(x,y)-W(x,y)\right)\right|\ge\varepsilon$. This means that either \begin{align}\label{eq:vetsiHustota}
\int_{x\in B_n}\int_{y\in B_n} \Gamma_n(x,y)&\ge\int_{x\in B_n}\int_{y\in B_n}W(x,y)+\varepsilon\quad\text{or}\\
\nonumber
\int_{x\in B_n}\int_{y\in B_n} \Gamma_n(x,y)&\le\int_{x\in B_n}\int_{y\in B_n}W(x,y)-\varepsilon\;.
\end{align}
By passing to a subsequence, we may assume that only one of these two cases occurs. We stick to the case when (\ref{eq:vetsiHustota}) holds for every natural number $n$ (the other case is analogous). By passing to a subsequence once again, we may assume that the sequence ${\bf 1}_{B_1},{\bf 1}_{B_2},{\bf 1}_{B_3},\ldots$ converges in the weak$^*$ topology to some $\psi\colon I\rightarrow[0,1]$. 	
For every natural number $n$, let $\mathcal J_n$ be the ordered partition of $I$ into two sets $B_n$ and $I\setminus B_n$ (in this order). This allows us to define $\alpha_{\mathcal J_n,1},\alpha_{\mathcal J_n,2},\gamma_{\mathcal J_n}\colon I\rightarrow I$ as in~\eqref{eq:pocitac}, and versions $\prescript{}{\mathcal J_1}\Gamma_1,\prescript{}{\mathcal J_2}\Gamma_2,\prescript{}{\mathcal J_3}\Gamma_3,\ldots$ of $\Gamma_1,\Gamma_2,\Gamma_3,\ldots$. We pass to a subsequence again to assure that the sequence $\prescript{}{\mathcal J_1}\Gamma_1,\prescript{}{\mathcal J_2}\Gamma_2,\prescript{}{\mathcal J_3}\Gamma_3,\ldots$ is convergent in the weak$^*$ topology, and we denote the weak$^*$ limit by $\widetilde W$.
Now Lemma~\ref{lem:neostraNerovnost} tells us that $\INT_f(\widetilde W)\le\INT_f(W)$, and that to prove that this inequality is sharp we only need to verify~\eqref{eq:TYR}. So to complete the proof, it suffices to prove the following claim.

\begin{claim}\label{cl:klavesnice}
We have
\begin{equation}\nonumber
\int_0^1\int_0^1\widetilde W(\theta(x),\theta(y))\psi(x)\psi(y)\ge\int_0^1\int_0^1W(x,y)\psi(x)\psi(y)+\tfrac 12\varepsilon\;.
\end{equation}
\end{claim}
\begin{proof}[Proof of Claim~\ref{cl:klavesnice}] We have
\begin{equation}\nonumber
\begin{aligned}
\int_0^1\int_0^1\widetilde W(\theta(x),\theta(y))\psi(x)\psi(y)
=&\int_0^{\theta(1)}\int_0^{\theta(1)}\widetilde W(x,y)\\
\JUSTIFY{$\prescript{}{\mathcal J_n}\Gamma_n\stackrel{w^*}{\rightarrow}\widetilde W$}=&\lim\limits_{n\rightarrow\infty}\int_0^{\theta(1)}\int_0^{\theta(1)}\prescript{}{\mathcal J_n}\Gamma_n(x,y)\\
\JUSTIFY{for large enough $n$, as $\alpha_{\mathcal J_n,1}(1)\rightarrow \theta(1)$}\ge&\limsup\limits_{n\rightarrow\infty}\int_0^{\alpha_{\mathcal J_n,1}(1)}\int_0^{\alpha_{\mathcal J_n,1}(1)}\prescript{}{\mathcal J_n}\Gamma_n(x,y)-\tfrac 12\varepsilon\\
\JUSTIFY{integration by substitution}=&\limsup\limits_{n\rightarrow\infty}\int_0^1\int_0^1\prescript{}{\mathcal J_n}\Gamma_n(\alpha_{\mathcal J_n,1}(x),\alpha_{\mathcal J_n,1}(y)){\bf 1}_{B_n}(x){\bf 1}_{B_n}(y)-\tfrac 12\varepsilon\\
\JUSTIFY{$\gamma_{\mathcal J_n}(x)=\alpha_{\mathcal J_n,1}(x)$ for every $x\in B_n$}=&\limsup\limits_{n\rightarrow\infty}\int_{B_n}\int_{B_n}\prescript{}{\mathcal J_n}\Gamma_n(\gamma_{\mathcal J_n}(x),\gamma_{\mathcal J_n}(y))-\tfrac 12\varepsilon\\
=&\limsup\limits_{n\rightarrow\infty}\int_{B_n}\int_{B_n}\Gamma_n(x,y)-\tfrac 12\varepsilon\\
\stackrel{(\ref{eq:vetsiHustota})}{\ge}&\limsup\limits_{n\rightarrow\infty}\int_{B_n}\int_{B_n}W(x,y)+\tfrac 12\varepsilon\\
\JUSTIFY{${\mathbf 1}_{B_n}\stackrel{w^*}{\rightarrow}\psi$}=&\int_0^1\int_0^1W(x,y)\psi(x)\psi(y)+\tfrac 12\varepsilon\;.
\end{aligned}
\end{equation}
\end{proof}

\begin{rem}
	The initial step when we ``shift the sets $B_n$ to the left'' crucially relies on the Euclidean order on $I$. This order is needless for the theory of graphons, i.e., graphons can be defined on a square of an arbitrary atomless separable probability space $\Omega$. A linear order on $\Omega$ can be always introduced additionally, as $\Omega$ is measure-isomorphic to $I$. So, while our results work in full generality for an arbitrary $\Omega$, we wonder if our argument can be modified so that the proof would naturally work without assuming a linear structure of the underlying probability space.
\end{rem}
\section{Proof of Theorem~\ref{thm:compactANDminimizer}$\ref{en:compact}$}\label{sec:proofCompact}
The bulk of the proof is given after proving the following key lemma.

\begin{lem}\label{lem:ACC=LIM}
For every sequence $\Gamma_1,\Gamma_2,\Gamma_3,\ldots:I^2\rightarrow[0,1]$ of graphons there exists a subsequence $\Gamma_{k_1},\Gamma_{k_2},\Gamma_{k_3},\ldots$ such that
\begin{equation}\nonumber
\inf\left\{\INT_f(W)\colon W\in\ACC(\Gamma_{k_1},\Gamma_{k_2},\Gamma_{k_3},\ldots)\right\}=\inf\left\{\INT_f(W)\colon W\in\LIM(\Gamma_{k_1},\Gamma_{k_2},\Gamma_{k_3},\ldots)\right\}\;.
\end{equation}
\end{lem}

\begin{proof}
We start by finding countably many subsequences $\mathcal S_1,\mathcal S_2,\mathcal S_3,\ldots$ of the sequence $\Gamma_1,\Gamma_2,\Gamma_3,\ldots$ such that for every natural number $n$ we have:
\begin{itemize}
\item[(i)] $\mathcal S_{n+1}$ is a subsequence of $\mathcal S_n$, and
\item[(ii)] there exists $W_{n+1}\in\LIM(\mathcal S_{n+1})$ such that
\begin{equation}\label{eq:hlad}
\INT_f(W_{n+1})<\inf\left\{\INT_f(W)\colon W\in\ACC(\mathcal S_n)\right\}+\tfrac 1n\;.
\end{equation}
\end{itemize}
This is done by induction. In the first step, we just define the sequence $\mathcal S_1$ to be the original sequence $\Gamma_1,\Gamma_2,\Gamma_3,\ldots$. Next suppose that we have already defined the subsequence $\mathcal S_n$ for some natural number $n$. Then there is a graphon $W_{n+1}\in\ACC(\mathcal S_n)$ such that
\begin{equation}\nonumber
\INT_f(W_{n+1})<\inf\left\{\INT_f(W)\colon W\in\ACC(\mathcal S_n)\right\}+\tfrac 1n\;.
\end{equation}
Now we find a subsequence $\mathcal S_{n+1}$ of $\mathcal S_n$ such that some versions of the graphons from $\mathcal S_{n+1}$ converge to $W_{n+1}$ in the weak$^*$ topology. This finishes the construction.

Now we use the diagonal method to define, for every natural number $n$, the graphon $\Gamma_{k_n}$ to be the $n$th element of the sequence $\mathcal S_n$. Then we have for every $n$ that
\begin{equation}\nonumber
\begin{aligned}
&\inf\left\{\INT_f(W)\colon W\in\ACC(\Gamma_{k_1},\Gamma_{k_2},\Gamma_{k_3},\ldots)\right\}\\
\JUSTIFY{$\Gamma_{k_n},\Gamma_{k_{n+1}},\Gamma_{k_{n+2}},\ldots$ is a subsequence of $\mathcal S_n$}\ge&\inf\left\{\INT_f(W)\colon W\in\ACC(\mathcal S_n)\right\}\\
\stackrel{\eqref{eq:hlad}}{>}&\INT_f(W_{n+1})-\tfrac 1n\\
\JUSTIFY{$W_{n+1}\in\LIM(\mathcal S_{n+1})\subset\LIM(\Gamma_{k_1},\Gamma_{k_2},\Gamma_{k_3},\ldots)$}\ge&\inf\left\{\INT_f(W)\colon W\in\LIM(\Gamma_{k_1},\Gamma_{k_2},\Gamma_{k_3},\ldots)\right\}-\tfrac 1n\;,
\end{aligned}
\end{equation}
and so
\begin{equation}\nonumber
\inf\left\{\INT_f(W)\colon W\in\ACC(\Gamma_{k_1},\Gamma_{k_2},\Gamma_{k_3},\ldots)\right\}\ge\inf\left\{\INT_f(W)\colon W\in\LIM(\Gamma_{k_1},\Gamma_{k_2},\Gamma_{k_3},\ldots)\right\}\;.
\end{equation}
The other inequality is trivial.
\end{proof}

We can now give the proof of Theorem~\ref{thm:compactANDminimizer}\ref{en:compact}.

By using Lemma~\ref{lem:ACC=LIM} and by passing to a subsequence, we may assume that \begin{equation}\nonumber
\inf\left\{\INT_f(W)\colon W\in\ACC(\Gamma_1,\Gamma_2,\Gamma_3,\ldots)\right\}=\inf\left\{\INT_f(W)\colon W\in\LIM(\Gamma_1,\Gamma_2,\Gamma_3,\ldots)\right\}\;.
\end{equation}
We construct the desired subsequence $\Gamma_{k_1},\Gamma_{k_2},\Gamma_{k_3},\ldots$ by the following construction.

In the first step, we find a graphon $W_1\in\LIM(\Gamma_1,\Gamma_2,\Gamma_3,\ldots)$ such that
\begin{equation}\nonumber
\INT_f(W_1)<\inf\left\{\INT_f(W)\colon W\in\ACC(\Gamma_1,\Gamma_2,\Gamma_3,\ldots)\right\}+1\;.
\end{equation}
By Lemma~\ref{lem:aproximacePrumerama}, there is a partition $\mathcal J_1$ of $I$ into finitely many intervals of positive measure such that $|\INT_f(W_1)-\INT_f(W_1^{\Join\mathcal J_1})|<1$. Then we clearly have
\begin{equation}\nonumber
\INT_f(W_1^{\Join\mathcal J_1})<\inf\left\{\INT_f(W)\colon W\in\ACC(\Gamma_1,\Gamma_2,\Gamma_3,\ldots)\right\}+2\;.
\end{equation}
By Lemma~\ref{lem:attainaveraged}, the graphon $W_1^{\Join\mathcal J_1}$ is also an element of the set $\LIM(\Gamma_1,\Gamma_2,\Gamma_3,\ldots)$, and so there is a sequence $\Gamma_1^1,\Gamma_2^1,\Gamma_3^1,\ldots$ of versions of $\Gamma_1,\Gamma_2,\Gamma_3,\ldots$ that converges to $\widetilde W_1:=W_1^{\Join\mathcal J_1}$ in the weak$^*$ topology. We define $\Gamma_{k_1}:=\Gamma_1$, and we also define a sequence $q_1^1,q_2^1,q_3^1,\ldots$ to be the increasing sequence of all natural numbers.

Now fix a natural number $n$ and suppose that we have already defined a finite subsequence $\Gamma_{k_1},\Gamma_{k_2},\ldots,\Gamma_{k_n}$ of $\Gamma_1,\Gamma_2,\Gamma_3,\ldots$. Suppose also that for every $1\le i\le n$, we have already constructed
\begin{itemize}
\item[(i)] a step-graphon $\widetilde W_i$ with steps given by some partition $\mathcal J_i$ of $I$ into finitely many intervals of positive measure such that $\mathcal J_i$ is a refinement of $\mathcal J_{i-1}$ (if $i>1$) and such that
\begin{equation}\nonumber
\INT_f(\widetilde W_i)<\inf\left\{\INT_f(W)\colon W\in\ACC(\Gamma_1,\Gamma_2,\Gamma_3,\ldots)\right\}+\tfrac 2i\;,
\end{equation}
and
\item[(ii)] an increasing sequence $q_1^i,q_2^i,q_3^i,\ldots$ of natural numbers which is a subsequence of $q_1^{i-1},q_2^{i-1},q_3^{i-1},\ldots$ (if $i>1$), together with a sequence $\Gamma_{q_1^i}^i,\Gamma_{q_2^i}^i,\Gamma_{q_3^i}^i,\ldots$ of versions of $\Gamma_{q_1^i},\Gamma_{q_2^i},\Gamma_{q_3^i},\ldots$ which converges to $\widetilde W_i$ in the weak$^*$ topology and such that (if $i>1$) for every natural number $j$ and for every intervals $K,L\in\mathcal J_{i-1}$ it holds that
\begin{equation}\nonumber
\int_K\int_L\Gamma_{q_j^i}^i(x,y)=\int_K\int_L\Gamma_{q_j^i}^{i-1}(x,y)\;.
\end{equation}
\end{itemize}
Then we find a graphon $\overline W_{n+1}\in\LIM(\Gamma_1,\Gamma_2,\Gamma_3,\ldots)$ such that
\begin{equation}\nonumber
\INT_f(\overline W_{n+1})<\inf\left\{\INT_f(W)\colon W\in\ACC(\Gamma_1,\Gamma_2,\Gamma_3,\ldots)\right\}+\tfrac{1}{n+1}\;.
\end{equation}
Find a sequence $\overline\Gamma_{q_1^n}^{n+1},\overline\Gamma_{q_2^n}^{n+1},\overline\Gamma_{q_3^n}^{n+1},\ldots$ of versions of $\Gamma_{q_1^n},\Gamma_{q_2^n},\Gamma_{q_3^n},\ldots$ which converges to $\overline W_{n+1}$ in the weak$^*$ topology.
For every natural number $j$, let $\phi_j:I\rightarrow I$ be the measure-preserving almost bijection satisfying $\overline\Gamma_{q_j^n}^{n+1}(x,y)=\Gamma_{q_j^n}^n(\phi_j^{-1}(x),\phi_j^{-1}(y))$ for a.e. $(x,y)\in I^2$ (such an almost-bijection exists as both $\overline\Gamma_{q_j^n}^{n+1}$ and $\Gamma_{q_j^n}^n$ are versions of the same graphon $\Gamma_{q_j^n}$).
Let us fix some order of the sets from the partition $\mathcal J_n$. For every $j$, let $\mathcal I_j$ be the ordered partition of $I$ consisting of the sets $\phi_j(K)$, $K\in\mathcal J_n$, with the order given by the order of the sets from $\mathcal J_n$.
Let $r_1,r_2,r_3,\ldots$ be a subsequence of $q_1^n,q_2^n,q_3^n,\ldots$ such that for every $K\in\mathcal J_n$, the sequence ${\textbf 1}_{\phi_1(K)},{\textbf 1}_{\phi_2(K)},{\textbf 1}_{\phi_3(K)},\ldots$ is convergent in the weak$^*$ topology.
Find an accumulation point $W_{n+1}$ of the sequence $\prescript{}{\mathcal I_1}{\overline\Gamma_{r_1}^{n+1}},\prescript{}{\mathcal I_2}{\overline\Gamma_{r_2}^{n+1}},\prescript{}{\mathcal I_3}{\overline\Gamma_{r_3}^{n+1}},\ldots$ (in the weak$^*$ topology).
By Corollary~\ref{cor:zmensiEntropii}, we have
\begin{equation}\nonumber
\INT_f(W_{n+1})\le\INT_f(\overline W_{n+1})<\inf\left\{\INT_f(W)\colon W\in\ACC(\Gamma_1,\Gamma_2,\Gamma_3,\ldots)\right\}+\tfrac{1}{n+1}\;.
\end{equation}
Let $s_1,s_2,s_3,\ldots$ be a subsequence of $r_1,r_2,r_3,\ldots$ such that the sequence $\prescript{}{\mathcal I_1}{\overline\Gamma_{s_1}^{n+1}},\prescript{}{\mathcal I_2}{\overline\Gamma_{s_2}^{n+1}},\prescript{}{\mathcal I_3}{\overline\Gamma_{s_3}^{n+1}},\ldots$ converges to $W_{n+1}$ in the weak$^*$ topology.
Note that for every natural number $j$ and for every intervals $K,L\in\mathcal J_n$, it holds that
\begin{equation}\label{SkoroIntegralyPresCtverecky}
\begin{aligned}
\int_K\int_L\prescript{}{\mathcal I_j}{\overline\Gamma_{s_j}^{n+1}}(x,y)=\int_{\phi_J(K)}\int_{\phi_j(L)}\overline\Gamma_{s_j}^{n+1}(x,y)&=\int_{\phi_J(K)}\int_{\phi_j(L)}\Gamma_{s_j}^n(\phi_j^{-1}(x),\phi_j^{-1}(y))\\
&=\int_K\int_L\Gamma_{s_j}^n(x,y)\;.
\end{aligned}
\end{equation}
By Lemma~\ref{lem:aproximacePrumerama}, there is a partition $\mathcal J_{n+1}$ of $I$ into finitely many intervals of positive measure such that $\mathcal J_{n+1}$ is a refinement of $\mathcal J_n$ and such that $|\INT_f(W_{n+1})-\INT_f(W_{n+1}^{\Join\mathcal J_{n+1}})|<\tfrac 1{n+1}$. Then we clearly have
\begin{equation}\nonumber
\INT_f(W_{n+1}^{\Join\mathcal J_{n+1}})<\inf\left\{\INT_f(W)\colon W\in\ACC(\Gamma_1,\Gamma_2,\Gamma_3,\ldots)\right\}+\tfrac 2{n+1}\;.
\end{equation}
By Lemma~\ref{lem:attainaveraged}, the graphon $\widetilde W_{n+1}:=W_{n+1}^{\Join\mathcal J_{n+1}}$ is a limit (in the weak$^*$ topology) of the sequence of some versions  $\Gamma_{s_1}^{n+1},\Gamma_{s_2}^{n+1},\Gamma_{s_3}^{n+1},\ldots$ of the graphons $\prescript{}{\mathcal I_1}{\overline\Gamma_{s_1}^{n+1}},\prescript{}{\mathcal I_2}{\overline\Gamma_{s_2}^{n+1}},\prescript{}{\mathcal I_3}{\overline\Gamma_{s_3}^{n+1}},\ldots$. By the ``moreover'' part of Lemma~\ref{lem:attainaveraged}, we may further assume that for every natural number $j$ and for every intervals $P,Q\in\mathcal J_{n+1}$, we have
\begin{equation}\nonumber
\int_P\int_Q\prescript{}{\mathcal I_j}{\overline\Gamma_{s_j}^{n+1}}(x,y)=\int_P\int_Q\Gamma_{s_j}^{n+1}(x,y)\;,
\end{equation}
which, together with (\ref{SkoroIntegralyPresCtverecky}), easily implies that for every natural number $j$ and for every intervals $K,L\in\mathcal J_n$ it holds
\begin{equation}
\int_K\int_L\Gamma_{s_j}^{n+1}(x,y)=\int_K\int_L\Gamma_{s_j}^n(x,y)\;.
\end{equation}
We define $\Gamma_{k_{n+1}}:=\Gamma_{s_{n+1}^{n+1}}$, and we also define the sequence $q_1^{n+1},q_2^{n+1},q_3^{n+1},\ldots$ to be the sequence $s_1,s_2,s_3,\ldots$. This completes the construction of the sequence $\Gamma_{k_1},\Gamma_{k_2},\Gamma_{k_3},\ldots$.

Now let $W_{\min}$ be an arbitrary accumulation point (in the weak$^*$ topology) of the sequence $\Gamma_{k_1},\Gamma_{k_2},\Gamma_{k_3},\ldots$, so that in particular $W_{\min}\in\ACC(\Gamma_1,\Gamma_2,\Gamma_3,\ldots)$. It suffices to show that it holds for every $n$ that $\INT_f(W_{\min})\le\INT_f(\widetilde W_n)$ as then we clearly have by our choice of the graphons $\widetilde W_1,\widetilde W_2,\widetilde W_3,\ldots$ that
\begin{equation}\nonumber
\INT_f(W_{\min})=\min\{\INT_f(W)\colon W\in\ACC(\Gamma_1,\Gamma_2,\Gamma_3,\ldots)\}\;.
\end{equation}
But for every three natural numbers $n<m$ and $j$ and for every intervals $K,L\in\mathcal J_n$ it holds by (ii) that
\begin{equation}\nonumber
\int_K\int_L\Gamma_{q_j^m}^m(x,y)=\int_K\int_L\Gamma_{q_j^m}^n(x,y)\;,
\end{equation}
and so (as $\Gamma_{q_j^n}^n\stackrel{w^*}{\rightarrow}\widetilde W_n$ as $j\rightarrow\infty$ for every $n$)
\begin{equation}\nonumber
\int_K\int_L\widetilde W_m(x,y)=\int_K\int_L\widetilde W_n(x,y)\;.
\end{equation}
It follows that for every $n$ it holds
\begin{equation}\nonumber
\int_K\int_LW_{\min}(x,y)=\int_K\int_L\widetilde W_n(x,y)\;.
\end{equation}
The rest follows by Lemma~\ref{lem:concavetriv}.

\section{Proof of Proposition~\ref{prop:cutnormMIN}}\label{sec:propcutnormMIN}
As promised, we give two proofs of Proposition~\ref{prop:cutnormMIN}. The first one is somewhat quicker, but uses a theorem of Borgs, Chayes, and Lov\'asz~\cite{MR2594615} about uniqueness of graph limits. More precisely, the theorem states that if $U':I^2\rightarrow[0,1]$ and $U'':I^2\rightarrow[0,1]$ are two cut-norm limits of versions $\Gamma_1',\Gamma_2',\Gamma_3',\ldots$ and $\Gamma_1'',\Gamma_2'',\Gamma_3'',\ldots$ of a graphon sequence $\Gamma_1,\Gamma_2,\Gamma_3,\ldots$, then there exists a graphon $U^*:I^2\rightarrow[0,1]$ that is a cut-norm limit of versions of $\Gamma_1,\Gamma_2,\Gamma_3,\ldots$, and measure preserving transformations $\psi',\psi'':I\rightarrow I$ such that for almost every $(x,y)\in I^2$, $U'(x,y)=U^*(\psi'(x),\psi'(y))$ and $U''(x,y)=U^*(\psi''(x),\psi''(y))$. Since then, the result was proven in several different ways, see~\cite[p.221]{Lovasz2012}. Also, let us note that while all known proofs of the Borgs--Chayes--Lov\'asz theorem are complicated, none uses the compactness of the space of graphons or the Regularity lemma. So, using this result as a blackbox, we still obtain a self-contained characterization of cut-norm limits in terms of weak$^*$ limits.

So, suppose that $W:I^2\rightarrow[0,1]$ is a limit of versions of $\Gamma_1,\Gamma_2,\Gamma_3,\ldots$ in the cut-norm. 
By Theorem~\ref{thm:compactANDminimizer} and by passing to a subsequence, we may assume that there exists a minimizer $W':I^2\rightarrow[0,1]$ of $\INT_f(\cdot)$ over $\LIM(\Gamma_1,\Gamma_2,\Gamma_3,\ldots)$ which is a limit of versions of $\Gamma_1,\Gamma_2,\Gamma_3,\ldots$ in the cut-norm.
Therefore, the Borgs--Chayes--Lov\'asz theorem tells us that there exists a graphon $W^*:I^2\rightarrow [0,1]$ and measure preserving maps $\psi,\psi':I\rightarrow I$ such that $W(x,y)=W^*(\psi(x),\psi(y))$ and $W'(x,y)=W^*(\psi'(x),\psi'(y))$ for almost every $(x,y)\in I^2$. Since $\psi$ and $\psi'$ are measure preserving, we get $\INT_f(W)=\INT_f(W^*)$ and $\INT_f(W')=\INT_f(W^*)$. This finishes the proof.\qed

\bigskip

Let us now give a self-contained proof of Proposition~\ref{prop:cutnormMIN}. 
By Theorem~\ref{thm:compactANDminimizer} and by passing to a subsequence, we may assume that there exists a minimizer $W':I^2\rightarrow[0,1]$ of $\INT_f(\cdot)$ over $\LIM(\Gamma_1,\Gamma_2,\Gamma_3,\ldots)$ which is a limit of versions $\Gamma'_1,\Gamma'_2,\Gamma'_3,\ldots$ of $\Gamma_1,\Gamma_2,\Gamma_3,\ldots$ in the cut-norm.
Suppose that $W$ is a graphon with $\INT_f(W)>\INT_f(W')$. This in particular means that there exists $\delta>0$ so that 
\begin{equation}\label{eq:sd}
\|W'-U\|_1>\delta
\end{equation}
for any version $U$ of $W$. We claim that there are no versions of $\Gamma_1,\Gamma_2,\Gamma_3,\ldots$ that converge to $W$ in the cut-norm. Indeed, suppose that such versions $\Gamma^*_1,\Gamma^*_2,\Gamma^*_3,\ldots$ exist. Observe that $\delta_1(\Gamma'_n,\Gamma^*_n)=0$ for each $n$ (in fact, the infimum in the definition of $\delta_1$ is attained). Now,  \cite[Lemma~2.11]{Lovasz}\footnote{Let us stress that \cite[Lemma~2.11]{Lovasz} does not rely on the Borgs--Chayes--Lov\'asz theorem, and has a self-contained, one-page proof.} tells us that 
$$0=\liminf_n 0=\liminf_n \delta_1(\Gamma'_n,\Gamma^*_n) \ge \delta_1(W',W)\;,
$$
which is a contradiction to~\eqref{eq:sd}.
\qed

\section{Concluding remarks}

\subsection{Specific concave and convex functions}
Perhaps the most natural choice of continuous concave function is the binary entropy $H$.

An equivalent characterization to our main result is that the limit graphons are the weak$^*$ limits that \emph{maximize} $\INT_g$ for a strictly \emph{convex} function $g$. The most interesting instance of this version of the statement is that the limit graphons are weak$^*$ limits maximizing the $L^2$-norm. Note that the $L^2$-norm is an infinitesimal counterpart to the notion of the ``index'' commonly used in proving the regularity lemma.

\subsection{Regularity lemmas as a corollary}\label{ssec:reglemmas}
While the cut-distance is most tightly linked to the weak regularity lemma of Frieze and Kannan~\cite{Frieze1999}, a short reduction given in~\cite{Lovasz2007} shows that Theorem~\ref{thm:convergeFormal} implies also Szemer\'edi's regularity lemma~\cite{Szemeredi1978}, and its ``superstrong'' form,~\cite{MR1804820}. So, it is possible to obtain these regularity lemmas using the approach from this paper.\footnote{With a notable drawback that we do not obtain any quantitative bounds.}

The most remarkable difference of the current approach is that it does not use iterative index-pumping, as we explained in Section~\ref{ssec:connectionwithRL}. That is, in our proof one refinement is sufficient for the argument. Such a shortcut is available only in the limit setting, it seems.

\subsection{A conjecture about finite graphs}\label{ssec:ConjectureFiniteGraphs} Let $f:[0,1]\rightarrow \mathbb R$ be an arbitrary continuous and strictly concave function. Suppose that $G$ is an $n$-vertex graph, and let $\mathcal P=(P_i)_{i=1}^k$ be a partition of $V(G)$ into non-empty sets. Recall the notion of $\INT_f(G;\mathcal P)$ and of densities $d_{ij}$ defined in Section~\ref{ssec:connectionwithRL}. We believe that a partition that minimizes $\INT_f(G;\cdot)$, when we range over all partitions $\mathcal Q$ of $G$ with a given (but large) number of parts, provides a good approximation of $G$ in the sense of the weak regularity lemma. To formulate this conjecture, let us say that a partition $\mathcal Q$ is \emph{an $\INT_f$-minimizing partition with $k$ parts} if $\mathcal Q$ has $k$ parts and for any partition $\mathcal P$ of $V(G)$ with $k$ parts we have $\INT_f(G;\mathcal Q)\le \INT_f(G;\mathcal P)$.
\begin{conjecture}\label{conj:finite}
	Suppose that $f:[0,1]\rightarrow \mathbb R$ is a continuous and strictly concave function, and that $\epsilon>0$ is given. Then there exist numbers $M,n_0$ so that the following holds for each graph $G$ of order at least $n_0$. If $\mathcal Q$ is an $\INT_f$-minimizing partition of $V(G)$ with $M$ parts then $\mathcal Q$ is also weak $\epsilon$-regular.
\end{conjecture}
This is a finite counterpart of our main result. Indeed, the space of weak* limits of graphons in Theorem~\ref{thm:compactANDminimizer} corresponds to an averaging over infinitesimally small sets, while in Conjecture~\ref{conj:finite} we range only over partitions with $M$ parts. Of course, the much finer partitions considered in Theorem~\ref{thm:compactANDminimizer} provide an ``$\epsilon=0$ error''. 

If true, Conjecture~\ref{conj:finite} would provide a more direct link between regularity and index-like parameters than the index-pumping lemma.

While we were not able to prove Conjecture~\ref{conj:finite}, let us present here a quick proof of a somewhat weaker statement.
\begin{prop}\label{prop:finiteIndexSet}
	Suppose that $f:[0,1]\rightarrow \mathbb R$ is a continuous and strictly concave function, and that $\epsilon>0$ is given. Then there exist a finite set $X\subset \mathbb{N}$ so that for each graph $G$ there exists $M\in X$ with the following property. If $\mathcal Q$ is an $\INT_f$-minimizing partition of $V(G)$ with $M$ parts then $\mathcal Q$ is also weak $\epsilon$-regular.
\end{prop}
Actually to prove Proposition~\ref{prop:finiteIndexSet}, one just needs to go through the proof of the weak regularity lemma. For simplicity, let us assume that $h:x\mapsto -x^2$ is the negative of the usual ``index'' used in the proof of the weak regularity lemma. Let us take $X:=\{1,2,4,\ldots,2^{\lceil\frac4{\epsilon^2}\rceil}\}$ . Suppose for a contradiction that for each $i\in X$, there is an $\INT_f$-minimizing partition $\mathcal C_i$ with $i$ parts which is weak $\epsilon$-irregular. Then Lemma~\ref{lem:indexpumping} assert that there exists a partition $\mathcal P_{i+1}$ with $2i$ parts such that $\INT_h(G;\mathcal P_{i+1})< \INT_h(G;\mathcal C_i)-\frac{\epsilon^2}4$. In particular, we have
\begin{align*}
0\ge \INT_h(G;\mathcal C_1)&> \INT_h(G;\mathcal P_{2})+\frac{\epsilon^2}4\ge 
\INT_h(G;\mathcal C_{2})+\frac{\epsilon^2}4\\
&>
\INT_h(G;\mathcal P_{3})+2\cdot\frac{\epsilon^2}4\ge
\INT_h(G;\mathcal C_{3})+2\cdot\frac{\epsilon^2}4>\ldots\\
&>
\INT_h(G;\mathcal P_{i+1})+i\cdot\frac{\epsilon^2}4\ge
\INT_h(G;\mathcal C_{i+1})+i\cdot\frac{\epsilon^2}4>\ldots\\
&>
\INT_h(G;\mathcal P_{\lceil\frac4{\epsilon^2}\rceil+1})+\lceil\frac4{\epsilon^2}\rceil\cdot\frac{\epsilon^2}4\;.
\end{align*}
This is a contradiction to the fact that $\INT_h(\cdot;\cdot)\ge -1$.

\medskip
One could consider even a ``stability version'' of Conjecture~\ref{conj:finite}. That is, it may be that if $\INT_f(G;\mathcal Q)$ is close to the minimum of $\INT_f(G;\mathcal P)$ over partitions $\mathcal P$ with $M$ parts, then $\mathcal Q$ is weak $\epsilon$-regular. For example, repeating the proof of Proposition~\ref{prop:finiteIndexSet} for a set $X=\{1,2,4,\ldots,2^{\lceil\frac8{\epsilon^2}\rceil}\}$, we get there exists $M\in X$ so that any partition $\mathcal Q$ with $M$ parts for which $$\INT_{x\mapsto -x^2}(G;\mathcal Q)\le \frac{\epsilon^2}{8}+\min\left\{\INT_{x\mapsto -x^2}(G;\mathcal P) \::\:\mbox{$\mathcal P$ has $M$ parts}\right\}$$
is weak $\epsilon$-regular.

\medskip
Also, Conjecture~\ref{conj:finite} could be asked for other versions of the regularity lemma.

\subsection{Attaining the infimum in Theorem~\ref{thm:compactANDminimizer}\ref{en:compact}}\label{ssec:needsubseqinfimum} Theorem~\ref{thm:compactANDminimizer}\ref{en:compact} states that there exist a subsequence of graphons $\Gamma_{k_1},\Gamma_{k_2},\Gamma_{k_3},\ldots$ such that the infimum of $\INT_f(\cdot )$ over the set $\ACC(\Gamma_{k_1},\Gamma_{k_2},\Gamma_{k_3},\ldots)$ is attained. Recently, Jon Noel showed us that passing to a subsequence is really needed. That is, taking $f$ to be the binary entropy function, he constructed a sequence of graphons $\Gamma_1,\Gamma_2,\Gamma_3,\ldots$ such that $\inf\{\INT_f(\Gamma):\Gamma\in \ACC(\Gamma_1,\Gamma_2,\Gamma_3,\ldots)\}=0$ but there exists no $\Gamma\in \ACC(\Gamma_1,\Gamma_2,\Gamma_3,\ldots)$ with $\INT_f(\Gamma)=0$. To this end, take $(W_\ell)_{\ell=1}^\infty$ to be rescaled adjacency matrices of a sequence of quasirandom graphs with edge density say $0.5$, but replacing in each adjacency matrix one diagonal element (now represented by a square $S_\ell$ of size $\frac{1}{\ell}\times \frac{1}{\ell}$) by value say $0.7$. Let $(\Gamma_n)_{n=1}^\infty$ be a  sequence in which each graphon $W_\ell$ occurs infinitely many times.

Firstly, we claim that $\inf\{\INT_f(\Gamma):\Gamma\in \ACC(\Gamma_1,\Gamma_2,\Gamma_3,\ldots)\}=0$. To see this, take $\ell$ large. Taking a subsequence $\Gamma_{k_1},\Gamma_{k_2},\Gamma_{k_3},\ldots$ which consists only of copies of $W_\ell$, we see that $W_\ell\in \ACC(\Gamma_1,\Gamma_2,\Gamma_3,\ldots)$. Now, $\INT_f(W_\ell)=\int_x\int_y f(W_\ell(x,y))\le \frac1{\ell^2}$, since the integrand is zero everywhere except $S_\ell$.

Secondly, we claim that there is no graphon in $\ACC(\Gamma_1,\Gamma_2,\Gamma_3,\ldots)$ with zero entropy. Indeed, let us consider a weak* limit $W$ of an arbitrary sequence of versions of $W_{\ell_1},W_{\ell_2},W_{\ell_3},\ldots$. There are two cases. If the sequence $\ell_1,\ell_2,\ell_3,\ldots$ is unbounded then quasirandomness of the graphons implies that $W\equiv \frac12$. The other case is when one index $\ell$ repeats infinitely many times. In that case, due to the value of $0.7$ on $S_\ell$, the graphon $W$ cannot be $\{0,1\}$-valued, as the next lemma shows.
\begin{lem}
Suppose that $\Lambda$ is an arbitrary probability measure space with a probability measure $\lambda$, and $\alpha>0$. Suppose that $(A_s)_{s=1}^\infty$ is a sequence of functions, $A_s:\Lambda\rightarrow[0,1]$, which converges weak* to a function $A$. Suppose further that $\lambda(R_s)\ge \alpha$ for each $s\in\mathbb{N}$, where $R_s:=\{x\in\Lambda: A_s=0.7\}$. Then $A$ is not $\{0,1\}$-valued.
\end{lem}
\begin{proof}
Suppose for a contradiction that $A$ is $\{0,1\}$-valued. Let $X_0=A^{-1}(0)$ and $X_1=A^{-1}(1)$. Then for each $s\in\mathbb{N}$, we have $\lambda(R_s\cap X_0)\ge \alpha/2$ or $\lambda(R_s\cap X_1)\ge \alpha/2$. Let us consider the case that the set $I_0$ of indices $s$ for which the former inequality occurs is infinite; the other case being analogous. For each $s\in I_0$ we have
$$\int_{X_0} A_s=\int_{X_0\cap R_s} A_s+\int_{X_0\setminus R_s} A_s\ge 0.7\cdot \lambda(X_0\cap R_s)+0\cdot \lambda(X_0\setminus R_s)\ge 0.7\cdot\frac\alpha2\;.$$
On the other hand, $\int_{X_0}A=0$. So, the set $X_0$ witnesses that the functions $(A_s)_{s\in I_0}$ do not weak* converge to $A$, a contradiction.
\end{proof}

In either of the two cases above, $W$ has positive entropy.

\subsection{Hypergraphs}
The theory of limits of dense hypergraphs of a fixed uniformity was worked out in~\cite{ElekSzegedy} (using ultraproduct techniques) and in~\cite{MR3382671} (using hypergraph regularity lemma techniques), and is substantially more involved. It seems that the current approach may generalize to the hypergraph setting. This is currently work in progress.

\subsection{Role of weak* limits for other combinatorial structures}
In this paper, we have shown how to use weak* limits for sequences of graphs to obtain cut-distance limits. In the section above we indicated that a similar approach may lead to a construction of limits of hypergraphs of fixed uniformity. Of course, one can ask which other limit concepts can be approached by considering weak* limits as an intermediate step. Let us point out that limits of permutations (\emph{permutons}) are particularly simple in this sense: Limits (in the ``cut-distance'' sense) of permutations arise simply by taking weak limits (here, it is weak rather than weak*, but the difference is not important) of certain objects associated directly to permutations. That is, no counterpart to our entropy minimization step is necessary, and every weak limit already has the desired combinatorial properties. See~\cite[Section~2]{MR3053756}. These are, to the best of our knowledge, the only combinatorial structures for which weak/weak* convergence was used.

\subsection{Minimization with respect to different concave functions}
Suppose that $f$ and $g$ are two different strictly concave functions. Then for two graphons $\Gamma_1$ and $\Gamma_2$, we can have for example $\INT_f(\Gamma_1)<\INT_f(\Gamma_2)$ but $\INT_g(\Gamma_1)>\INT_g(\Gamma_2)$. As a (perhaps somewhat surprising) by-product of our main results, we cannot get such an inconsistency when searching global minima over the space of weak* limits. That is, $\Gamma_1$ achieves the minimum of $\INT_f$ on the space of weak* limits if and only if it achieves the minimum of $\INT_g$. We do not know of a more direct proof of this fact.

\subsection{Recent developments}After this paper was made available at arXiv in May 2017, the relation between the cut distance and the weak* topology was studied in more detail in~\cite{DGHRR:WeakStarAbstract} and~\cite{DGHRR:Parameters}. The main novel feature in~\cite{DGHRR:WeakStarAbstract} is an abstract approach which allows to identify convergent subsequences and cut distance limits without minimization of any parameter over the space of weak* limits. The main two theorems in~\cite{DGHRR:WeakStarAbstract} which were inspired by the present paper are the following:
\begin{thm}\label{thm:Abstract1}
Suppose that $\Gamma_1,\Gamma_2,\Gamma_3,\ldots:I^2\rightarrow[0,1]$ is a sequence of graphons. Then there exists a subsequence $\Gamma_{k_1},\Gamma_{k_2},\Gamma_{k_3},\ldots$  such that
	\begin{equation*}\ACC(\Gamma_{k_1},\Gamma_{k_2},\Gamma_{k_3},\ldots)=\LIM(\Gamma_{k_1},\Gamma_{k_2},\Gamma_{k_3},\ldots)\;.
	\end{equation*}
\end{thm}
\begin{thm}
	Suppose that $\Gamma_1,\Gamma_2,\Gamma_3,\ldots:I^2\rightarrow[0,1]$ is a sequence of graphons. Then this sequence is cut-distance convergent if and only if
	\begin{equation*}\ACC(\Gamma_{1},\Gamma_{2},\Gamma_{3},\ldots)=\LIM(\Gamma_{1},\Gamma_{2},\Gamma_{3},\ldots)\;.
	\end{equation*}
\end{thm}
In particular, note that Theorem~\ref{thm:Abstract1} substantially generalizes Lemma~\ref{lem:ACC=LIM}. Actually, investigating possible generalizations of Lemma~\ref{lem:ACC=LIM} was the starting point for~\cite{DGHRR:WeakStarAbstract}.

\medskip 
Besides this abstract approach,  some further graphon parameters that can replace $\INT_f(\cdot)$ in Theorem~\ref{thm:compactANDminimizer} are found in~\cite{DGHRR:Parameters}. These include, for example, the negative of the density of any even cycle, $-t(C_{2\ell},\cdot)$. On the other hand, in another very recent paper, Kr\'al', Martins, Pach, and Wrochna~\cite{KMPW:StepSidorenko} identify a large class of (bipartite) graphs $H$ for which $-t(H,\cdot)$ fails to identify cut distance limits. The problem of characterizing graphs $H$ which this property is related to the Sidorenko conjecture and to norming graphs motivated by a question of Lov\'asz and studied first in~\cite{Hat:Siderenko}.

Also, the machinery introduced in~\cite{DGHRR:Parameters} gives a short proof of a version of Theorem~\ref{thm:compactANDminimizer} which even allows to drop the requirement on the continuity of $f$.

\section*{Acknowledgements}
This work was done while Jan Hladk\'y was enjoying a lively atmosphere of the Institute for Geometry at TU Dresden, being hosted by Andreas Thom there.

We thank Dan Kr\'al and Oleg Pikhurko for encouraging conversations on the subject, Jon Noel for comments on an earlier version of the manuscript, and to Svante Janson and Guus Regts for bringing several important references to our attention. We also thank Jon Noel for his contribution included in Section~\ref{ssec:needsubseqinfimum}.

Finally, we thank two anonymous referees for their comments, and in particular, for pointing out a gap in the proof of Lemma~\ref{lem:neostraNerovnost}.

\appendix
\section{The weak$^*$ topology}
Suppose that $X$ is a Banach space and denote by $X^*$ its dual. Then the weak$^*$ topology on $X^*$ is the coarsest topology on $X^*$ such that all mappings of the form $X^*\ni x^*\mapsto x^*(x)$, $x\in X$, are continuous. Recall that if the space $X$ is separable then by the sequential Banach--Alaoglu Theorem (see e.g.~\cite[Theorem 1.9.14]{Tao:EpsilonRoomI}), the unit ball of $X^*$ is sequentially compact. This means that every bounded sequence of elements of the dual space $X^*$ contains a weak$^*$-convergent subsequence.

In this paper, we are interested in the case when $X$ is the Banach space $L^1(\Omega)$ of all integrable functions on some probability space $\Omega$. (Depending on our needs, the probability space $\Omega$ will be chosen to be either the unit interval $I$ equipped with the one-dimensional Lebesgue measure or the unit square $I^2$ equipped with the two-dimensional Lebesgue measure). The space $L^1(\Omega)$ is equipped with the norm $\|f\|_1=\int_{\Omega}|f(x)|$, $f\in L^1(\Omega)$. In this setting, the dual $X^*=(L^1(\Omega))^*$ is isometric to the space $L^{\infty}(\Omega)$ of all bounded measurable functions on $\Omega$, equipped with the norm $\|g\|_{\infty}=\text{ess}\sup_{x\in\Omega}|g(x)|$.
The duality between $L^1(\Omega)$ and $L^{\infty}(\Omega)$ is given by the formula $\langle g,f\rangle=\int_{\Omega}f(x)g(x)$ for $g\in L^{\infty}(\Omega)$ and $f\in L^1(\Omega)$.
This means that a sequence $g_1,g_2,g_3,\ldots$ of elements of $L^{\infty}(\Omega)$ converges to $g\in L^{\infty}(\Omega)$ if and only if $\lim_{n\rightarrow\infty}\int_{\Omega}f(x)g_n(x)=\int_{\Omega}f(x)g(x)$ for every $f\in L^1(\Omega)$.

Now consider the Banach space $X=L^1(I^2)$ of all integrable functions defined on the unit square $I^2$ (which is equipped with the two-dimensional Lebesgue measure).
Standard arguments show that the weak$^*$ topology on its dual space $L^{\infty}(I^2)$ can be equivalently generated by mappings of the form $L^\infty(I^2)\ni g\mapsto\int_A\int_Bg(x,y)$ where $A,B$ are measurable subsets of $I$. That is, the weak$^*$ topology can be equivalently generated only by characteristic functions of measurable rectangles (instead of all integrable functions on $I^2$).
If we restrict this topology only to the space of all graphons $W:I^2\rightarrow[0,1]$ defined on $I^2$ then it is easy to see that this restricted topology is generated only by mappings of the form $W\mapsto\int_A\int_AW(x,y)$ where $A$ is a measurable subset of $I$ (this is because each graphon is symmetric by the definition).
This is the topology we refer to when we talk about convergence of graphons in the weak$^*$ topology.
So this means that a sequence $W_1,W_2,W_3,\ldots$ of graphons defined on $I^2$ converges to a graphon $W$ defined on $I^2$ if and only if $\lim_{n\rightarrow\infty}\int_A\int_AW_n(x,y)=\int_A\int_AW(x,y)$ for every measurable subset $A$ of $I$.
Note that the space of all graphons defined on $I^2$ is a weak$^*$ closed subset of the unit ball of $L^{\infty}(I^2)$, and so it is sequentially compact by the sequential Banach--Alaoglu Theorem (as the space $L^1(I^2)$ is separable).

While crucial to our arguments, it is worth noting that the Banach--Alaoglu Theorem is not a particularly deep statement and follows easily from Tychonoff's theorem for powers of compact spaces (and actually the version for countable powers is sufficient).
\bibliographystyle{plain}
\bibliography{bibl.bib}

\begin{thebibliography}{10}

\bibitem{MR637937}
D.~J. Aldous.
\newblock Representations for partially exchangeable arrays of random
  variables.
\newblock {\em J. Multivariate Anal.}, 11(4):581--598, 1981.

\bibitem{MR883646}
D.~J. Aldous.
\newblock Exchangeability and related topics.
\newblock In {\em \'Ecole d'\'et\'e de probabilit\'es de {S}aint-{F}lour,
  {XIII}---1983}, volume 1117 of {\em Lecture Notes in Math.}, pages 1--198.
  Springer, Berlin, 1985.

\bibitem{MR1804820}
N.~Alon, E.~Fischer, M.~Krivelevich, and M.~Szegedy.
\newblock Efficient testing of large graphs.
\newblock {\em Combinatorica}, 20(4):451--476, 2000.

\bibitem{MR2426176}
T.~Austin.
\newblock On exchangeable random variables and the statistics of large graphs
  and hypergraphs.
\newblock {\em Probab. Surv.}, 5:80--145, 2008.

\bibitem{MR2594615}
C.~Borgs, J.~Chayes, and L.~Lov\'asz.
\newblock Moments of two-variable functions and the uniqueness of graph limits.
\newblock {\em Geom. Funct. Anal.}, 19(6):1597--1619, 2010.

\bibitem{Borgs2008c}
C.~Borgs, J.~T. Chayes, L.~Lov{\'a}sz, V.~T. S{\'o}s, and K.~Vesztergombi.
\newblock {Convergent sequences of dense graphs. {I}. {S}ubgraph frequencies,
  metric properties and testing}.
\newblock {\em Adv. Math.}, 219(6):1801--1851, 2008.

\bibitem{ChatVar:LargeDev}
S.~Chatterjee and S.~R.~S. Varadhan.
\newblock {The large deviation principle for the {E}rd\H{o}s-{R}{\'e}nyi random
  graph}.
\newblock {\em European J. Combin.}, 32(7):1000--1017, 2011.

\bibitem{MR2463439}
P.~Diaconis and S.~Janson.
\newblock Graph limits and exchangeable random graphs.
\newblock {\em Rend. Mat. Appl. (7)}, 28(1):33--61, 2008.

\bibitem{DGHRR:Parameters}
M.~Dole{\v z}al, J.~Greb{\'i}k, J.~Hladk{\'y}, I.~Rocha, and V.~Rozho{\v n}.
\newblock Cut distance identifying graphon parameters over weak* limits.
\newblock arXiv:1809.03797.

\bibitem{DGHRR:WeakStarAbstract}
M.~Dole{\v z}al, J.~Greb{\'i}k, J.~Hladk{\'y}, I.~Rocha, and V.~Rozho{\v n}.
\newblock Relating the cut distance and the weak* topology for graphons.
\newblock arXiv:1809.03797.

\bibitem{ElekSzegedy}
G.~Elek and B.~Szegedy.
\newblock A measure-theoretic approach to the theory of dense hypergraphs.
\newblock {\em Adv. Math.}, 231(3-4):1731--1772, 2012.

\bibitem{Frieze1999}
A.~Frieze and R.~Kannan.
\newblock {Quick Approximation to Matrices and Applications}.
\newblock {\em Combinatorica}, 19(2):175--220, 1999.

\bibitem{Hat:Siderenko}
H.~Hatami.
\newblock {Graph norms and {S}idorenko's conjecture}.
\newblock {\em Israel J. Math.}, 175:125--150, 2010.

\bibitem{KMPW:StepSidorenko}
D.~Kr\'al', T.~Martins, P.~P. Pach, and M.~Wrochna.
\newblock {The step Sidorenko property and non-norming edge-transitive graphs}.
\newblock {\em J. Combin. Theory Ser. A}, 162:34--54, 2019.

\bibitem{MR3053756}
D.~Kr\'al' and O.~Pikhurko.
\newblock Quasirandom permutations are characterized by 4-point densities.
\newblock {\em Geom. Funct. Anal.}, 23(2):570--579, 2013.

\bibitem{Lovasz2012}
L.~Lov{\'a}sz.
\newblock {\em {Large networks and graph limits}}, volume~60 of {\em {American
  Mathematical Society Colloquium Publications}}.
\newblock American Mathematical Society, Providence, RI, 2012.

\bibitem{Lovasz2006}
L.~Lov{\'a}sz and B.~Szegedy.
\newblock {Limits of dense graph sequences}.
\newblock {\em J. Combin. Theory Ser. B}, 96(6):933--957, 2006.

\bibitem{Lovasz2007}
L.~Lov{\'a}sz and B.~Szegedy.
\newblock {Szemer{\'e}di's {L}emma for the analyst}.
\newblock {\em J. Geom. and Func. Anal}, 17:252--270, 2007.

\bibitem{Lovasz}
L.~Lov{\'a}sz and B.~Szegedy.
\newblock Testing properties of graphs and functions.
\newblock {\em Israel J. Math.}, 178:113--156, 2010.

\bibitem{McDiarmid1989}
C.~McDiarmid.
\newblock {On the method of bounded differences}.
\newblock In {\em {Surveys in combinatorics, 1989 ({N}orwich, 1989)}}, volume
  141 of {\em {London Math. Soc. Lecture Note Ser.}}, pages 148--188. Cambridge
  Univ. Press, Cambridge, 1989.

\bibitem{MR3425986}
G.~Regts and A.~Schrijver.
\newblock Compact orbit spaces in {H}ilbert spaces and limits of edge-colouring
  models.
\newblock {\em European J. Combin.}, 52(part B):389--395, 2016.

\bibitem{Scott}
A.~Scott.
\newblock {Szemer{\'e}di's regularity lemma for matrices and sparse graphs}.
\newblock {\em Combin. Probab. Comput.}, 20(3):455--466, 2011.

\bibitem{Szemeredi1978}
E.~Szemer{\'e}di.
\newblock {Regular partitions of graphs}.
\newblock In {\em {Probl{\`e}mes combinatoires et th{\'e}orie des graphes
  (Colloq. Internat. CNRS, Univ. Orsay, Orsay, 1976)}}, volume 260 of {\em
  {Colloq. Internat. CNRS}}, pages 399--401. CNRS, Paris, 1978.

\bibitem{Tao:EpsilonRoomI}
T.~Tao.
\newblock {\em An epsilon of room, {I}: real analysis}, volume 117 of {\em
  Graduate Studies in Mathematics}.
\newblock American Mathematical Society, Providence, RI, 2010.
\newblock Pages from year three of a mathematical blog.

\bibitem{MR3382671}
Y.~Zhao.
\newblock Hypergraph limits: a regularity approach.
\newblock {\em Random Structures Algorithms}, 47(2):205--226, 2015.

\end{thebibliography}
\end{document}